\documentclass{amsart}

\usepackage[utf8]{inputenc}
\usepackage[T1]{fontenc}
\usepackage{verbatim}
\usepackage{graphicx}
\usepackage{graphicx,caption2,psfrag,float,color}
\usepackage{amssymb}
\usepackage{amscd}
\usepackage{amsmath}

\usepackage{tikz}
\usetikzlibrary{positioning}

\usepackage{mathabx}

\usepackage[T1]{fontenc}

\newtheorem{theorem}{Theorem}[section]

\newtheorem{lemma}[theorem]{Lemma}
\newtheorem{proposition}[theorem]{Proposition}
\numberwithin{equation}{subsection}
\newtheorem{definition}[theorem]{Definition}

\title{The ternary Goldbach problem with two Piatetski-Shapiro primes and a prime with a missing digit}
\author{Helmut Maier and Michael Th. Rassias}
\date{\today}
\address{Department of Mathematics, University of Ulm, Helmholtzstrasse 18, 89081 Ulm, Germany.}
\email{helmut.maier@uni-ulm.de}
\address{Institute of Mathematics, University of Zurich, CH-8057, Zurich, Switzerland 
\&
Moscow Institute of Physics and Technology,
141700 Dolgoprudny,
Institutskiy per, d. 9,
Russia 
\& Institute for Advanced Study, Program in Interdisciplinary Studies,
1 Einstein Dr, Princeton, NJ 08540, USA.}\email{michail.rassias@math.uzh.ch}\thanks{}
\begin{document}

 \maketitle
 
\begin{abstract} 
Let 
$$\gamma^*=\frac{8}{9}+\frac{2}{3}\:\frac{\log(10/9)}{\log 10}\:(\approx 0.919\ldots)\:.$$
Let $\gamma^*<\gamma_0\leq 1$, $c_0=1/\gamma_0$ be fixed. Let also $a_0\in\{0,1,\ldots, 9\}$.\\
We prove on assumption of the Generalized Riemann Hypothesis that each sufficiently large odd integer $N_0$ can be represented in the form 
$$N_0=p_1+p_2+p_3\:,$$
where the $p_i$ are of the form $p_i=[n_i^{c_0}]$, $n_i\in\mathbb{N}$, for $i=1,2$ and the decimal expansion of $p_3$ does not contain the digit $a_0$.\\
The proof merges methods of J. Maynard from his paper on the infinitude of primes with restricted digits, results of A. Balog and J. Friedlander on Piatetski-Shapiro primes and the Hardy-Littlewood circle method in two variables. This is the first result on the ternary Goldbach problem with primes of mixed type which involves primes with missing digits. \\ \\
 \textbf{Key words.} Ternary Goldbach problem; Generalized Riemann Hypothesis; Hardy-Littlewood circle method; Piatetski-Shapiro primes; primes with missing digit.\\ 
\textbf{2010 Mathematics Subject Classification:} 11P32, 11N05, 11A63.%
\newline
\end{abstract}

\section{Introduction and statement of result}

The ternary Goldbach problem was treated by Vinogradov \cite{vino} (see also \cite{rasvino}):\\
Let
\[
R(N_0)=\sum_{\substack{(p_1,p_2,p_3)\\ p_1+p_2+p_3=N_0}} (\log p_1)(\log p_2)(\log p_3) \tag{1.1}
\]
(Here and in the sequel the letter $p$ denotes primes).\\
Then
\[
R(N_0)=\frac{1}{2}\:\mathfrak{S}(N_0)N_0^2+O_A\left(\frac{N_0^2}{\log^A N_0}\right)\:, \tag{1.2}
\]
for arbitrary $A>0$, where  $\mathfrak{S}(N_0)$ is the singular series
\[
\mathfrak{S}(N_0)=\prod_{p\mid N_0}\left(1-\frac{1}{(p-1)^2}\right)\prod_{p\nmid N_0}\left(1+\frac{1}{(p-1)^3}\right)\:. \tag{1.3}
\]
The relation (1.2) implies that each sufficiently large odd integer is the sum of three primes.\\
Helfgott \cite{helfgott} recently showed that this is true for all odd $N\geq 7$. Later on, solutions of
$$p_1+p_2+p_3=N_0$$
with the $p_i$ taken from special sets $S_i$ were investigated.\\
Piatetski-Shapiro \cite{Piatetski} proved that for any fixed $c_0\in(1,12/11)$ the sequence $([n^{c_0}])_{n\in\mathbb{N}}$ contains infinitely many prime numbers. The interval for $c_0$ was sebsequently improved many times (cf. \cite{Deshouillers}, \cite{Heath-Brown}, \cite{Kolenski},  \cite{Rivat}).
Balog and Friedlander \cite{balog} considered the ternary Goldbach problem with variables restricted to Piatetski-Shapiro primes. They proved that for any fixed $c_0$ with $1<c_0<21/20$ every sufficiently large odd integer $N_0$ can be  represented in the form 
$$N_0=p_1+p_2+p_3\:, \ \ \text{with $p_i=[n_i^{c_0}]$}\:,$$
for any $n_i\in\mathbb{N}$.\\
For other combinations of the sets $S_i$, cf. \cite{balog}, \cite{jia}, \cite{Kumchev}, \cite{Teravainen}. Additionally, the authors in \cite{maier_rassias_vinogradov} proved that under the assumption of the Generalized Riemann Hypothesis each sufficiently large odd integer can be expressed as the sum of a prime and two isolated primes.\\
The other type of primes entering our hybrid theorem are primes with missing digits.\\
In many papers numbers with restricted digits have been investigated (cf. \cite{Banks1}, \cite{Banks2}, \cite{Bourgain}, \cite{Dartyge1}, \cite{Dimitrov1},  \cite{Dimitrov2},  \cite{Dartyge2}, \cite{Drmota}, \cite{Erdos1}, \cite{Erdos2}, \cite{Konyagin}, \cite{Mauduit}). These investigations culminated in the work of Maynard \cite{maynard}, who proved the existence of infinitely many primes with restricted digits. In \cite{maynard} Maynard proved the following:\\
\textit{Let $a_0\in\{0,1,\ldots,9 \}$. Then there are infinitely many primes, whose decimal expansion does not contain the digit $a_0$.}\\
Our result merges methods of J. Maynard from \cite{maynard}, results of A. Balog and J. Friedlander on Piatetski-Shapiro primes \cite{balog} and the Hardy-Littlewood circle method in two variables. This is the first result on the ternary Goldbach problem with primes of mixed type which involves primes with missing digits. We shall prove the following:

\begin{theorem}\label{thm11}
Assume the Generalized Riemann Hypothesis (GRH). Let 
$$\gamma^*=\frac{8}{9}+\frac{2}{3}\:\frac{\log(10/9)}{\log 10}\:(\approx 0.919\ldots)\:.$$
Let $\gamma^*<\gamma_0\leq 1$, $c_0=1/\gamma_0$ be fixed. Let also $a_0\in\{0,1,\ldots, 9\}$.\\
Then each sufficiently large odd integer $N_0$ can be represented in the form 
$$N_0=p_1+p_2+p_3\:,$$
where the $p_i$ are of the form $p_i=[n_i^{c_0}]$, $n_i\in\mathbb{N}$, for $i=1,2$ and the decimal expansion of $p_3$ does not contain the digit $a_0$.\\

\end{theorem}

\section{Outline and some basic definitions}

Our argument is closely related to the work of Maynard \cite{maynard}. Important ideas use a sieve decomposition of the counting function $\#\{p\in\mathcal{A}\}$, which is based on ideas of Harman \cite{harman}, as well as the discrete circle method.\\
We recall the following definitions from \cite{maynard} which we complement by a few new definitions.

\begin{definition}\label{def21}
Let $a_0\in\{0,1,\ldots,9\}$, $k\in\mathbb{N}$ and let
\begin{align*}
&\mathcal{A}:=\left\{  \sum_{0\leq i\leq k} n_i 10^i \::\: n_i\in\{0,1,\ldots,9\}\setminus \{a_0\}\right\} \:,\\
&X:=10^k\:,\ \ \mathcal{B}:=\{n\leq X,\ n\in\mathbb{N}\}\:,
\end{align*}
\begin{align*}
&\mathbb{P}\ \text{the set of prime numbers},\\
&S_{\mathcal{A}}(\theta):=\sum_{a\in\mathcal{A}} e(a\theta)\:,\ \ 
S_{\mathbb{P}}(\theta):=\sum_{p\leq x}e(p\theta)\:,\ \ 
S_{\mathcal{A}\cap\mathbb{P}}(\theta):=\sum_{p\in \mathcal{A}\cap \mathbb{P}}e(p\theta)\:.
\end{align*}
Let $\mathcal{C}$ be a set of integers. We define the characteristic function $1_{\mathcal{C}}$ by
\begin{eqnarray}
1_{\mathcal{C}}(n):=\left\{ 
  \begin{array}{l l}
    1\:, & \quad \text{if $n\in \mathcal{C}$}\vspace{2mm}\\ 
    0\:, & \quad \text{if $n\not\in \mathcal{C}$}\:.\\
  \end{array} \right.
\nonumber
\end{eqnarray}
For $d\in\mathbb{N}$ we set
$$\mathcal{C}_d:=\{c\::\: cd\in\mathcal{C}\}.$$
The sifted set $\mathcal{U}(\mathcal{C}, z)$ is defined by
$$\mathcal{U}(\mathcal{C}, z):=\{c\in\mathcal{C}\::\: p\mid c\ \Rightarrow\ p>z\}\:.$$
The sieving function $S(\mathcal{C}, z)$ - the counting function of $\mathcal{U}(\mathcal{C}, z)$ - is given by
$$S(\mathcal{C}, z):=\#\mathcal{U}(\mathcal{C}, z)=\#\{c\in \mathcal{C}\::\: p\mid c\ \Rightarrow \ p>z\}\:.$$
We let
\begin{eqnarray}
w_n:=1_{\mathcal{A}}(n)-\frac{\kappa_{\mathcal{A}}\#\mathcal{A}}{\#\mathcal{B}}\:,\ \ \ 
\kappa_{\mathcal{A}}:=\left\{ 
  \begin{array}{l l}
    \dfrac{10(\Phi(10)-1)}{9\Phi(10)}\:, & \quad \text{if $(10,a_0)=1$}\vspace{2mm}\\ 
    \dfrac{10}{9}\:, & \quad \text{otherwise}\:,\\
  \end{array} \right.
\nonumber
\end{eqnarray}
$$S_d(z):=\sum_{\substack{n<X/d\\ p\mid n\ \Rightarrow\ p>z}} w_{nd}=S(\mathcal{A}_d, z)-\frac{\kappa_{\mathcal{A}} \#\mathcal{A}}{X}\: S(\mathcal{B}_d, z)\:,$$
$$1_{\mathcal{A}}(n)\ \ \text{is called the $\mathcal{A}-$part of } w_n\:,$$
$$S(\mathcal{A}_d, z)\ \ \text{is called the $\mathcal{A}-$part of } S_d(z)\:,$$
the $\mathcal{B}$-parts are defined analogously.\\
We also define the exponential sums 
$$S(\mathcal{C}, z, \theta):=\sum_{n\in\mathcal{U}(\mathcal{C},z)}e(n\theta)\:,$$
$$S_d(z, \theta):=\sum_{\substack{n<X/d\\ p\mid n\ \Rightarrow\ p>z}}w_{nd}e(n\theta)=S(\mathcal{A}_d,z,\theta)-\frac{\kappa_{\mathcal{A}}\#\mathcal{A}}{X}S(\mathcal{B}_d, z, \theta)\:,\ \ (\theta\in\mathbb{R})\:.$$
\end{definition}

The essential idea of Harman's sieve is contained in Harman \cite{harman}, Theorem 3.1 from \cite{harman} (The Fundamental Theorem).\\
Suppose that for any sequences of complex numbers, $a_m, b_n$, that satisfy $|a_m|\leq 1$, $|b_n|\leq 1$ we have for some $\lambda>0$, $\alpha>0$, $\beta\leq 1/2$, $M\geq 1$ that 
\[
\sum_{\substack{mn\in\mathcal{A}\\ m\leq M}} a_m=\lambda \sum_{\substack{mn\in\mathcal{B}\\ m\leq M}}a_m+O(Y) \tag{3.3.1}
\]
and
\[
\sum_{\substack{mn\in\mathcal{A}\\ X^{\alpha}\leq m\leq X^{\alpha+\beta}}} a_mb_n=\lambda \sum_{mn\in\mathcal{B}}a_mb_n+O(Y)\:, \tag{3.3.2}
\]
where $Y$ is a suitably chosen constant.\\
Let $c_r$ be a sequence of complex numbers, such that $|c_r|\leq 1$, and if $c_r\neq 0$, then 
\[
p\mid r\ \Rightarrow\ p>x^{\epsilon}\:,\ \text{for some}\ \epsilon>0.  \tag{3.3.3}
\]
Then, if $X^{\alpha}<M$, $2R<\min(X^\alpha, M)$ and $M>X^{1-\alpha}$, if $2R>X^{\alpha+\beta}$, we have
\[
\sum_{r\sim R}c_rS(\mathcal{A}_r, X^\beta)=\lambda\sum_{r\sim R} c_r S(B_r, X^\beta)+O(Y\log^3X)\:. \tag{3.3.4}
\]
The equation (3.3.1) is known as type I information, whereas as (3.3.2) is known as type II information.\\
In the application of Theorem 3.1, information about a (complicated) set $\mathcal{A}$ is obtained from that of a (simple) set $\mathcal{B}$.\\
In Maynard's paper \cite{maynard} the sets $\mathcal{A}$ and $\mathcal{B}$ are those from Definition \ref{def21}. In a first step of the sieve decomposition the counting-function of interest $\#\{p\in \mathcal{A}\}$ is broken up as follows:
\begin{align*}
\#\{p\in \mathcal{A}\}&=\#\{p\in\mathcal{A}\::\: p>X^{1/2}\}+O(X^{1/2})\\
&=S_1(z_4)+(1+o(1))\:\frac{\kappa_{\mathcal{A}}\#\mathcal{A}}{\log X}\ \ \text{(here $z_4=X^{1/2}$)}\:.
\end{align*}
The function $S_1$ is now replaced in a series of steps by other terms of the form $S_d$. These steps consist in the application of \textit{Buchstab's recursion}:\\
Let $u_1< u_2$. Then 
\[
S(\mathcal{C}, u_2)=S(\mathcal{C}, u_1)-\sum_{u_1<p\leq u_2}S(\mathcal{C}_p, p)\:. \tag{2.1}
\]
In the transformation of the form
$$S_d(z)=S(\mathcal{A}_d, z)-\frac{\kappa_{\mathcal{A}}\#\mathcal{A}}{\log X}\:S(\mathcal{B}_d, z)\:,$$
(2.1) is now applied separately with $C=\mathcal{A}_d$ and $C=\mathcal{B}_d$ and we get the recursion:
$$S_1(u_2)=S_1(u_1)-\sum_{u_1<p\leq u_2}S_p(p)\:.$$
An important observation is that 
``the counting function version'' of the Buchstab recursion is linked to a ``characteristic function version''.
\[
1_{\mathcal{U}(\mathcal{C}, u_2)}(n)= 1_{\mathcal{U}(\mathcal{C}, u_1)}(n)-\sum_{u_1<p\leq u_2}1_{\mathcal{U}(\mathcal{C}_p, p)}(n)\:.  \tag{2.2}
\]
In our paper the discrete circle method is applied and therefore we multiply (2.2) with the exponential function $e(n\theta)\:(=e^{2\pi i n\theta})$, to obtain:
\[
1_{\mathcal{U}(\mathcal{C}, u_2)}(n)e(n\theta)=1_{\mathcal{U}(\mathcal{C}, u_1)}(n)e(n\theta)-\sum_{u_1<p\leq u_2}1_{\mathcal{U}(\mathcal{C}_p, p)}(n) e(n\theta)\:.  \tag{2.3}
\]
We get the following version of Buchstab's recursion, which we state as
\begin{lemma}\label{lem21}
Let $u_1<u_2$. Then 
\[
S(\mathcal{C}, u_2, \theta)=S(\mathcal{C}, u_1, \theta)-\sum_{u_1<p\leq u_2}S(\mathcal{C}_p, p, \theta)\:. \tag{2.4}
\]
\end{lemma}

We introduce another modification in our paper. Instead of considering all the integers in $\mathcal{A}$ as possible candidates for our representation of $N_0$ we now only choose the integers from a subset $\mathcal{A}^*$ of $\mathcal{A}$, which are contained in a short subinterval of $\mathcal{B}$.

\begin{definition}\label{rdef23}
Let $H\in \mathbb{N}$, $H\leq k$. For 
$$n=\sum_{j=1}^k n_j 10^j,\ \ (n_j\in \{0, \ldots, 9\})$$
we write
$$n_{H,1}:= \sum_{j=k-H+1}^k n_j 10^j=: \tilde{n}_H\cdot 10^{k-H+1}$$
and
$$n_{H, 2}:= \sum_{j=0}^{k-H} n_j 10^j\:.$$
\end{definition}

\begin{lemma}\label{rlem24}
Let $n=\tilde{n}_H\cdot 10^{k-H+1}$ as in Definition \ref{rdef23}. Then
\[
n\in \mathcal{A}\ \ \text{if and only if}\ \ \tilde{n}_H\in \mathcal{A}\ \ \text{and}\ \ n_{H,2}\in\mathcal{A}\tag{2.1}
\]
There is an integer $\tilde{n}_H\in \mathcal{A}\cap [0, 10^{H-1}]$ such that for $n^*_H:= \tilde{n}_H 10^{k-H+1}$ we have the following:
\[
\left|n^*_H-5\cdot 10^{k-1}\right|\leq \frac{3}{2}\: 10^{k-2}  \tag{2.2}
\]
and for $n_{H, 2}\in \mathcal{B}^*:= [ n_H^*, n_H^*+10^{k-H} )$ the following holds:
\[
n^*_H+n_{H,2}\in \mathcal{A}\ \ \Rightarrow\ \ n_{H, 2}\in \mathcal{A}\:.  \tag{2.3}
\]
\end{lemma}
\begin{proof}
(2.1) is obvious. To show (2.2) and (2.3) we consider the following cases:\\
\textit{Case 1}: $a_0=5$, \textit{Case 2}: $a_0=4$, \textit{Case 3}: $a_0\not\in \{4, 5\}$.\\
Specifically we have:\\
\noindent \textit{Case 1}: Let $n_{k-i}\in \{0, \ldots, 9\}\setminus\{a_0\}$ for $2\leq i\leq H-1$. Then we may take
$$n^*=4,9\cdot 10^k+\sum_{j=k-H+1}^{k-2} n_j10^j\:.$$

\noindent \textit{Case 2}: Let  $n_{k-i}\in \{0, \ldots, 9\}\setminus\{a_0\}$ for $3\leq i\leq H-1$. Then we may take
$$n^*=5,09\cdot 10^k+\sum_{j=k-H+1}^{k-3} n_j10^j\:.$$
\noindent \textit{Case 3}: The choices for $n^*$ in cases 2 and 3 are both possible.
\end{proof}

\textit{Convention:} In the sequel we have many estimates and definitions containing positive constants $C_1, C_2, \ldots$ (actually powers $(\log X)^{C_i}$ ). The $C_i$ must satisfy certain conditions, which will be described. However, it will always be possible to choose the $C_i$, such that the $\min_{i} C_i$ is arbitrarily large. An estimate containing $O\left(D(x)(\log X)^{-\mathcal{A}}\right)$ ($D(x)$ a certain function of $X$) means that $A>0$ may be taken arbitrarily large if $\min_i C_i$ is sufficiently large.

\begin{definition}\label{rdef25}
We define $X$ by $2X\leq N_0< 20 X$. We then define 
\[
Int(N_0)=\left[ \frac{N_0-n_H^*}{2}-\frac{X}{8},\  \frac{N_0-n_H^*}{2}+\frac{X}{8} \right]\:, \tag{2.4}
\]
\[
S_{c_0}(\theta):=\frac{1}{\gamma}\ \sum_{\substack{p \in Int(N_0) \\ p=[n^{1/\gamma}]}} (\log  p)^{1-\gamma} e(p\theta)\:.   
\]
\end{definition}

Let $S\subseteq [1, X]$ be a set of positive integers and $v(n)$ be a sequence of real numbers. For the exponential sum
\[ 
E(\theta):= \sum_{n\in S} v(n) e(n\theta)  \tag{2.5}
\]
we define 
\[
J(E):= \frac{1}{X} \sum_{1\leq a \leq X} E\left( \frac{a}{X} \right) S_{c_0}^2\left( \frac{a}{X} \right) e\left(-N_0\frac{a}{X}\right)  \tag{2.6}
\]
\[
J(E, \tau):= \frac{1}{X} \sum_{\frac{a}{X}\in \mathcal{T}} E\left( \frac{a}{X} \right) S_{c_0}^2\left( \frac{a}{X} \right) e\left(-N_0\frac{a}{X}\right)  
\]
for a subset $\mathcal{T}\subseteq [0,1]$ and the mean-value
$$M(E):=\sum_{\substack{(m,p_2, p_3)\\ m\in S, p_2, p_3\in \mathbb{P}_{c_0}, p_i\in Int(N_0)\\ m+p_2+p_3=N_0}} p_2^{1-\gamma} p_3^{1-\gamma} (\log p_2)(\log p_3)v(m) \:,$$
The evaluation of $J(E)$ is also called the \textit{$a$-variable} circle method.

\begin{lemma}\label{rlem26}
We have
$J(E)=M(E)$.
\end{lemma}
\begin{proof}
This follows by orthogonality.
\end{proof}
Instead of $w_n$, $S_d(z)$ from Definition $\ref{def21}$ we now consider the expression given in 
\begin{definition}\label{rdef27}
We determine $H$ by $10^H=\lceil(\log X)^{C_1}\rceil$. Let $n^*=n^*_H$, which has been constructed in Lemma \ref{rlem24}, $\mathcal{B}^*$ as in Lemma \ref{rlem24} and $\mathcal{A}^*=\mathcal{A}\cap \mathcal{B}^*$.\\
We let
$$w_n^*:= 1_{\mathcal{A}^*}(n)-\frac{\kappa_{\mathcal{A}}\#\mathcal{A}^*}{\#\mathcal{B}^*}$$
$$S^*_d(z):=\sum_{\substack{n< X/d \\ p\mid n\ \Rightarrow p>z}} w^*_{nd}=S(\mathcal{A}_d^*, z)-\frac{\kappa_{\mathcal{A}}\#\mathcal{A}^*}{\#\mathcal{B}^*} S(B_d^*, z)$$
$1_{\mathcal{A}_d^*}(n)$ is called the $\mathcal{A}$-part of $w_n^*$.\\
$S(\mathcal{A}_d^*, z)$ is called the $\mathcal{A}$-part  of $S^*_d(z)$.\\
The $\mathcal{B}$-parts are defined analogously. The analogue of Lemma \ref{lem21} leads to an identity.
$$S_{\mathcal{A}^*\cap \mathbb{P}}(\theta)=\sum_j E_j(\theta)\:,$$
where the exponential sums $E_j(\theta)$ are extended over integers $n=p_1\ldots p_l$, defined by linear inequalities to be satisfied by the vector 
$$\left(\frac{\log p_1}{\log X}\:,\ \ldots, \frac{\log p_l}{\log X} \right)\:.$$
We also define the exponential sums
$$S_d^*(z, \theta)=\sum_{\substack{n\in \mathcal{B}^* \\ p\mid n\ \Rightarrow\ p>z}} w^*_{nd}= S(\mathcal{A}_d^*, z, \theta)-\frac{\kappa_{\mathcal{A}}\#\mathcal{A}^*}{\#\mathcal{B}^*} S(\mathcal{B}_d^*, z, \theta)\:.$$
For the evaluation of the sums $J(E)$ from (2.6) by the \textit{$a$-variable} circle method we partition the set $\{\frac{a}{X}\::\: 1\leq a\leq X\}$ into the two subsets of the major arcs and the minor arcs.
\end{definition}

\begin{definition}\label{rdef28}
We set $Q_0=(\log X)^3$. For $q\leq X$, $1\leq  c\leq q$, $(c, q)=1$ and $L\in [1, \infty)$ we set:
$$I_{c,q}(L):=\left[  \frac{c}{q}-q^{-1} X^{-1}L, \frac{c}{q}+q^{-1} X^{-1}L\right]\:.$$
We let $L_0=(\log X )^{C_1}$, $L_1=X^{1/5}$. The major arcs $\mathcal{M}$ are defined as 
$$\mathcal{M}:=\bigcup_{\substack{q\leq Q_0\\ (c,q)=1}} I_{c,q}(L_0)\:.$$
The minor arcs  $\mathfrak{m}$ are defined as
$$\mathfrak{m}:=[0,1]\setminus \mathcal{M}\:.$$
\end{definition}

For the evaluation of $S_{c_0}\left(\frac{c}{q}+\xi\right)$ we apply the approach of Balog and Friedlander \cite{balog}.\\
We now obtain a \textit{local version} of the result of Maynard \cite{maynard}. Instead of considering the sets $C_d$ with $C=\mathcal{A}$ and $\mathcal{B}$ appearing in the Buchstab recursion in Lemma \ref{lem21} we now consider the sets $C_d$ with 
$$C^*=C_{q,s}:=\{m\in C\::\: m\equiv s\bmod q\}\:,$$
where $C=\mathcal{A}^*$ or $\mathcal{B}^*$ as defined in Definition \ref{rdef27}.\\
We carry out the type I and type II estimates closely following Maynard \cite{maynard}, obtaining the contributions to $J(E)$ of the major arcs of the $a$-variable circle method. These type II estimates are based on the $b$-variable circle method:\\
Let $R$ be a subset of $\mathcal{B}^*$, $\mathcal{J}=\mathcal{A}^*\cap R$. Then we have
$$S_{\mathcal{J}}(\theta)=\frac{1}{X}\sum_{1\leq b\leq X}S_{\mathcal{A}}\left(\frac{b}{X}\right) S_R\left(-\left(\frac{b}{X}-\theta\right)\right)\:.$$
The minor arcs of the $a$-variable circle method finally are treated by estimates of Large Sieve type and by estimates of exponential sums over prime numbers.


\section{Structure of the paper}

In Section 4 we carry out the sieve decomposition of the local version of Maynard \cite{maynard} involving the exponential sums  instead of counting functions and the sets $\mathcal{A}^*$ and $\mathcal{B}^*$ contained in short intervals.\\
We shall reduce the proof of Theorem \ref{thm11} to the proof of three Propositions:\\ 
Proposition \ref{prop41} our type I estimate, Proposition \ref{prop42} our type II estimate and Proposition \ref{prop43} in which the $\mathcal{A}$-part is estimated trivially. All these propositions contain convolutions of the sums appearing in the Buchstab iterations with the Piatetski-Shapiro sums.\\
In Section 6 we reduce Propositions \ref{prop41}, \ref{prop42} and \ref{prop43} to the local version of Maynard's result, Propositions \ref{prop63} and \ref{prop64}, which do not involve the Piatetski-Shapiro sum.\\
Proposition \ref{prop63} is handled by a method from combinatorial sieve theory, replacing the M\"obius function by functions with smaller support and Fourier analysis to fix locations and residue-classes. \\
The proof of Proposition \ref{prop42} is carried out by the Classical Circle Method.\\
In Section 7 the ranges of summation are partitioned in small boxes.\\
These are now handled by the $b$-variable Circle Method, closely following Maynard \cite{maynard}. The dependency graph between the main statements is as follows:\\

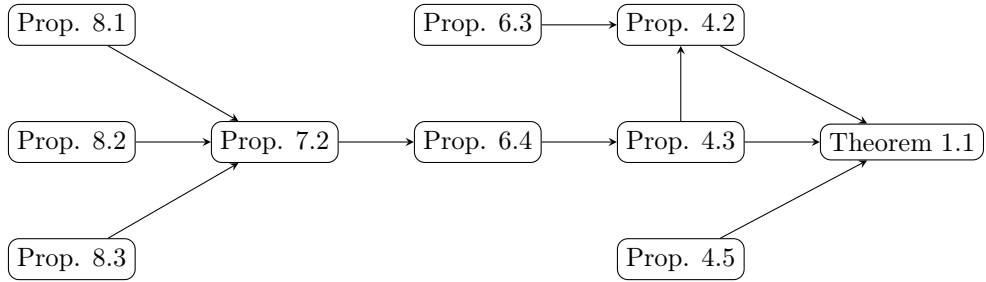
 \begin{tikzpicture}[>=stealth,every node/.style={shape=rectangle,draw,rounded corners},]
    \node (c1) {Prop. \ref{prop81}};
    \node (c2) [below =of c1]{Prop. \ref{prop82}};
    \node (c3) [below =of c2]{Prop. \ref{prop83}};
    \node (c4) [right=of c2]{Prop. \ref{prop72}};
        \node (cr1) [right=of c4]{Prop. \ref{prop64}};
                 \node (c9) [right =of cr1]{Prop. \ref{prop42}};
     \node (c8) [below =of c9]{Prop. \ref{prop43}};
      \node (c10) [above =of c9]{Prop. \ref{prop41}};
            \node (cr2) [left =of c10]{Prop. \ref{prop63}};

      \node (c11) [right =of c9]{Theorem \ref{thm11}};

    \draw[->] (c1) to (c4);
        \draw[->] (c2) to  (c4);
            \draw[->] (c3) to  (c4);
                       \draw[->] (c4) to  (cr1);                    
                       \draw[->] (cr1) to  (c9);  
                             \draw[->] (cr2) to  (c10);  

                       \draw[->] (c9) to  (c11);
                         \draw[->] (c8) to  (c11);
                          \draw[->] (c10) to  (c11);
                            \draw[->] (c9) to  (c10);
\end{tikzpicture}

\section{Sieve decomposition and proof of Theorem \ref{thm11}}

Here we carry out the modification of Maynard's method of Sieve Decomposition as described in the outline and reduce the proof of Theorem \ref{thm11} to the proof of Propositions \ref{prop41}, \ref{prop42} and \ref{prop43}.\\
Proposition \ref{prop41} deals with convolutions of exponential sums of type I, Proposition \ref{prop42} with those of exponential sums of type II, whereas Proposition \ref{prop43} gives a result in the case in which neither type I nor type II information is available.

\begin{definition}
Let $\eta\in (0,1)$. Let $v(n, \eta)_{n\in\mathcal{S}}$ be a family of sequences of real numbers, indexed by the parameter $\eta$, $\mathcal{S}$ finite. The family of exponential sums 
$$E(\theta; \eta):=\sum_{n\in\mathcal{S}}v(n, \eta)e(n\theta)$$
is called \textit{negligible}, if 
$$\lim_{\eta\rightarrow 0}\limsup_{k\rightarrow \infty} \frac{|J(E)|\log X}{(\#\mathcal{A}^*)X}=0$$
the term $``negligible"$ will also be applied to an individual exponential sum $E(\theta)$ of the family $E(\theta, \eta)$.
\end{definition}

\begin{proposition}\label{prop41} (Sieve asymptotic terms)\\
Let $\epsilon>0$, $0<\eta_0\leq \theta_2-\theta_1$, $l=l(\eta_0)$ be fixed, where 
$$\theta_1=\frac{9}{25}+2\epsilon\ \  \text{and}\ \  \theta_2=\frac{17}{40}-2\epsilon.$$
Let $\mathcal{L}$ be a set of $O_{\eta_0}(1)$ affine linear functions, $L\::\: \mathbb{R}^l\rightarrow \mathbb{R}$. Let
$$E_0:=E_0(\theta, \eta_0)=\sum_{X^{\eta_0}\leq p_1\leq \cdots \leq p_l}^{\sim} S_{p_1\cdots p_l}^*(X^{\eta_0},\theta)\:,$$
where $\sum^\sim$ indicates that the summation is restricted by the condition 
$$L\left( \frac{\log p_1}{\log X},\ldots, \frac{\log p_l}{\log X}\right)\geq 0\:,$$
for all $L\in\mathcal{L}$.\\
Then $E_0$ is negligible.
\end{proposition}


\begin{proposition}\label{prop42} (Type II terms)\\
Let $l=l(\eta_0), \theta_1, \theta_2, \mathcal{L}$, be as in Proposition \ref{prop41} and let $\mathcal{I}=\{1,\ldots, l\}$ and \mbox{$j\in\{1,\ldots, l\}$,}
$$E_1(\theta, \eta_0):=\sum_{\substack{X^{\eta_0}\leq p_1\leq \cdots \leq p_l \\ X^{\theta_1}\leq \prod_{i\in I}p_i\leq X^{\theta_2}\\ p_1\cdots p_l\leq X/p_j}}^\sim  S_{p_1\cdots p_l}^*(p_j,\theta)\:,$$
$$E_2(\theta, \eta_0):=\sum_{\substack{X^{\eta_0}\leq p_1\leq \cdots \leq p_l \\ X^{1-\theta_2}\leq \prod_{i\in I}p_i\leq X^{1-\theta_1}\\ p_1\cdots p_l\leq X/p_j}}^\sim  S_{p_1\cdots p_l}^*(p_j,\theta)\:,$$
where $\sum^\sim$ indicates the same restriction as in Proposition \ref{prop41}.\\
Then $E_1$ and $E_2$ are negligible.
\end{proposition}

\begin{definition}\label{def42}
The Buchstab function $\omega$ is defined by the delay-differential equation
$$\omega(u)=\frac{1}{u}\:,\ 1\leq u\leq 2\:,$$
$$\omega'(u)=\omega(u-1)-\omega(u)\:,\ \ u>2\:.$$
For $\vec{p}=(p_1,\ldots,p_l)$, $p_i$ primes for $1\leq i \leq l$, let
$$Log(\vec{p})=\left(\frac{\log p_1}{\log X},\ldots, \frac{\log p_l}{\log X}\right)\:.$$
Let $\mathcal{C}$ be a set of $O(1)$ affine linear functions. Let the polytope $\mathcal{R}$ be defined by
$$\mathcal{R}=\{(u_1, \ldots, u_l)\in[0,1]^l\::\: L(u_1,\ldots, u_l)\geq 0\ \ \text{for all}\ L\in\mathcal{L}\}\:.$$
Let 
$$\Pi(\vec{p})=p_1\cdots p_l\:,$$
$$\mathfrak{S}(N_0)=\prod_{p\nmid N_0}\left(1+\frac{1}{(p-1)^3}\right)\prod_{p\mid N_0}\left(1-\frac{1}{(p-1)^2}\right)\:.$$\end{definition}

\begin{proposition}\label{prop43}
Let $l\in\mathbb{N}$, $\delta>0$,
$$z\::\: [0,1]^l\rightarrow [\delta, 1-\delta],\ \vec{u}=(u_1,\ldots, u_l)\rightarrow z(\vec{u})=z(u_1, \ldots, u_l)$$
be continuous.
Let 
$$E(\theta):=\sum_{\vec{p}\::\: Log(\vec{p})\in \mathcal{R}} S(\mathcal{B}^*_{\prod(\vec{p})}, X^{z(Log(\vec{p}))}, \theta)\:.$$
Then 
$$J(E(\theta))=\frac{X(\#\mathcal{B}^*)}{4\log X}\: \mathfrak{S}_0(N_0) \int\cdots\int_\mathcal{R} \frac{\omega(1-u_1-\cdots - u_l)}{u_1\cdots u_l z(u_1, \ldots, u_l)}\: du_1\ldots du_l\: (1+o(1))\:.$$
\end{proposition}

We now closely follow Maynard \cite{maynard} to decompose the exponential sum
$$S_{\mathcal{A}\cap \mathbb{P}}(\theta)=E^*(\theta):=\sum_{p\in \mathcal{A}} e(p\theta)\:.$$
We recall the following notations from \cite{maynard}. We let
$$z_1\leq z_2 \leq z_3\leq z_4\leq z_5\leq z_6$$
be given by
\begin{align*}
&z_1:=X^{\theta_2-\theta_1}\:,\ \ z_2:=X^{\theta_1}\:,\ \ z_3:=X^{\theta_2}\\
&z_4:=X^{1/2}\:,\ \ z_5:=X^{1-\theta_2}\:,\ \ z_6:=X^{1-\theta_1}\:.
\end{align*}

We write $N_i$ for negligible sums, $P_i$ for sums with $J(P_i)\geq 0$ and, given a positive constant $I:=E_i(I,\theta)$ for an exponential sum with
$$J(E_i(I,\theta))=\Gamma\cdot I(1+o(1))\:.$$
We have by Lemma \ref{rlem26}

$$S_{\mathcal{A}\cap\mathbb{P}}(\theta)=\sum_{p>X^{1/2}}e(p\theta)+N_1(\theta)=S_1(z_4, \theta)+\frac{\kappa_{\mathcal{A}}\#\mathcal{A}}{\log X}S_{\mathcal{B}}(\theta)+N_1(\theta)\:.$$

By Buchstab's identity we have:
$$S_1(z_4, \theta)=S_1(z_1, \theta)-\sum_{z_1<p\leq z_2} S_p(p, \theta)\:.$$

$S_1(z, \theta)$ is negligible by Proposition \ref{prop41}. We split the sum over $p$ into ranges $(z_i, z_{i+1})$ and see that all the terms with $p\in (z_2, z_3)$ are also negligible by Proposition \ref{prop42}. This gives
$$S_1^*(z_4, \theta)=-\sum_{z_1<p\leq z_2} S_p^*(p,\theta)-\sum_{z_3<p\leq z_4} S_p^*(p,\theta)+N_2(\theta)\:.$$
We wish to replace $S_p^*(p,\theta)$ by 
$$S_p^*\left(\min\left(p, \left(\frac{X}{p}\right)^{1/2}\right), \theta\right)\:.$$
We note, that these are the same, when $p\leq X^{1/3}$, but if $p>X^{1/3}$, then there are additional terms in 
$$S_p^*\left(\left(\frac{X}{p}\right)^{1/2}, \theta\right).$$
The existence of two primes $q_1, q_2$ with 
$$\left(\frac{x}{p}\right)^{1/2}<q_1<q_2\leq p$$
leads to the contradiction $pq_1q_2>X$.\\
Thus we have with $\delta:=1/(\log X)^{1/2}$
\begin{align*}
-P_1(\theta)&+\sum_{p<X^{1/2}}\left(S\left(\mathcal{A}_p, \min\left(p, \left(\frac{X}{p}\right)^{1/2}\right), \theta\right)-S^*(\mathcal{A}_p, p, \theta)\right)\\
&=\sum_{p< X^{\frac{1}{2}-\delta}}\sum_{\substack{\left(\frac{X}{p}\right)^{1/2}<q\leq p \\ qp\in \mathcal{A}}} e(pq\theta)+\sum_{X^{\frac{1}{2}-\delta}\leq p\leq X^{\frac{1}{2}}} S^*(\mathcal{A}_p, z_1, \theta)-P_2(\theta)\\
&=N_3(\theta)\:,  \tag{4.1}
\end{align*}
by Lemma \ref{rlem26}.\\
Similarly, we get corresponding bounds for 
$$S^*\left(\mathcal{B}_p, \min\left(p, \left(\frac{X}{p}\right)^{1/2}\right),\theta\right)$$
and so we can replace $S^*_p(p, \theta)$ by 
$$S_p^*\left(\min\left(p, \left(\frac{X}{p}\right)^{1/2}, \theta\right)\right)$$
at the cost of  negligible sum.\\
Using this, and applying Buchstab's identity again, we have:
\begin{align*}
S_1^*(z_4, \theta)&=-\sum_{z_1<p\leq z_2} S_p^*\left(\min\left(p, \left(\frac{X}{p}\right)^{1/2}, \theta\right)\right)\\
&-\sum_{z_3<p\leq z_4}S_p^*\left(\min\left(p, \left(\frac{X}{p}\right)^{1/2}, \theta\right)\right) +N_4(\theta)\\
&=-\sum_{z_1<p\leq z_2} S_p^*(z_1, \theta)-\sum_{z_3<p\leq z_4} S_p^*(z_1, \theta) + \sum_{\substack{z_1<1\leq p\leq z_2 \\ q\leq \left(\frac{X}{q}\right)^{1/2}}}S_{pq}^*(q, \theta)\\
&+\sum_{\substack{z_3<1\leq p\leq z_4 \\ z_1<q\leq \left(\frac{X}{q}\right)^{1/2}}}S_{pq}^*(q, \theta)+N_5(\theta)\:.
\end{align*}

The first two terms above are asymptotically negligible by Proposition \ref{prop41} and so this simplifies to
\[
S_1^*(z_4, \theta)=\sum_{\substack{z_1<q\leq p\leq z_2 \\ q\leq \left(\frac{X}{p}\right)^{1/2}}} S_{pq}^*(q, \theta)+ \sum_{\substack{z_3<q\leq p\leq z_4 \\ z_1<q\leq \left(\frac{X}{p}\right)^{1/2}}} S_{pq}^*(q, \theta) + N_6(\theta)\:. \tag{4.2}
\]
We perform further decompositions to the remaining terms in (6.5). We first concentrate on the first term on the right hand side. Splitting the ranges of $pq$ into intervals, and recalling those with a $pq$ in the interval $[z_2, z_3]$ or $[z_5, z_6]$ make a negligible contribution by Proposition \ref{prop42}, we obtain
\begin{align*}
\sum_{\substack{z_1<q\leq p\leq z_2 \\ q\leq \left(\frac{X}{p}\right)^{1/2}}} S_{pq}^*(q, \theta) =& \sum_{\substack{z_1<q\leq p\leq z_2 \\ q\leq \left(\frac{X}{p}\right)^{1/2} \\ z_0<pq}} S_{pq}^*(q, \theta)
+\sum_{\substack{z_1<q\leq p\leq z_2 \\ q\leq \left(\frac{X}{p}\right)^{1/2} \\ z_3\leq pq<z_5}} S_{pq}^*(q, \theta)\\ 
&+\sum_{\substack{z_1<q\leq p\leq z_2 \\ z_1\leq pq <z_2}} S_{pq}^*(q, \theta) + N_6(\theta)\:.   \tag{4.3}
\end{align*}

Here we have dropped the condition $q\leq (X/p)^{1/2}$ in the final sum, since this is implied by $q\leq p$ and $pq\leq z_2$. On recalling the definition of $w_n$, we can write:
$$\sum_{\substack{z_1<q\leq p\leq z_2 \\ q\leq \left(\frac{X}{p}\right)^{1/2} \\ z_6<pq}} S_{pq}^*(q, \theta)=P_3(\theta)-\frac{\kappa_{\mathcal{A}}\#\mathcal{A}^*}{\log X}\: \sum_{\substack{z_1<q\leq p\leq z_2 \\ q\leq \left(\frac{X}{p}\right)^{1/2} \\ z_6<pq}} S_{pq}^*(q, \theta)  S(\mathcal{B}_{pq}, q, \theta)\:.$$
By Proposition \ref{prop43} we obtain:
\[
\sum_{\substack{z_1<q\leq p\leq z_2 \\ q\leq \left(\frac{X}{p}\right)^{1/2} \\ z_0<pq}} S_{pq}^*(q, \theta)= \frac{\kappa_{\mathcal{A}}\#\mathcal{A}^*}{\log X} E_1(I_1, \theta)\:,   \tag{4.4}
\]
where
\[
I_1=\iint_{\substack{\theta_2-\theta_1<v\leq \theta_1 \\ v<(1-u)/2 \\ 1-\theta_1< u+v}} \omega\left(\frac{1-u-v}{v}\right)\: \frac{dudv}{uv^2} \:. \tag{4.5}
\]

$I_1$ is the first one of  a series of nine integrals $I_1, \ldots, I_9$, which occur in Maynard \cite{maynard}.\\
We perform further decompositions of the second term of (4.3), first splitting according to the size of $q^2p$ compared with $z_6$.
\[
\sum_{\substack{z_1<q\leq p\leq z_2 \\ q\leq \left(\frac{X}{p}\right)^{1/2} \\ z_3\leq pq< z^5}} S_{pq}^*(q, \theta)= \sum_{\substack{z_1<q\leq p\leq z_2 \\ z_3\leq pq< z_5 \\ q^2p<z_6}} S_{pq}^*(q, \theta) +  \sum_{\substack{z_1<q\leq p< z_2 \\ z_3\leq pq< z_5 \\ z_6\leq q^2p\leq X}} S_{pq}^*(q, \theta) \:. \tag{4.6}
\]

For the second term of (4.6), when $q^2p$ is large, we first separate the contribution from products of three primes: By an essentially identical argument to when we replaced $S_p^*(p, \theta)$ by 
$$S_p^*\left( \min\left(p, \left(\frac{X}{p}\right)^{1/2}\right), \theta\right)$$
we may replace $S_{pq}^*(q, \theta)$ by 
$$S_{pq}^*\left( \min\left(q, \left(\frac{X}{pq}\right)^{1/2}\right), \theta\right)$$
at the cost of a negligible sum $N_7(\theta)$ (since $pq<z_6$).\\
By Buchstab's identity  we have (with $r$ restricted to being prime):
\begin{align*}
&\sum_{\substack{z_1<q\leq p\leq z_2 \\ z_3\leq pq<z_5 \\ z_6\leq q^2p\leq  X}} S_{pq}^*\left( \min\left(p, \left(\frac{X}{p}\right)^{1/2}\right), \theta\right)=\\ 
&\sum_{\substack{z_1\leq q< p\leq z_2 \\ z_3\leq pq<z_5 \\ z_6\leq q^2p\leq  X}} S_{pq}^*\left(  \left(\frac{X}{pq}\right)^{1/2}, \theta\right)
+
\sum_{\substack{z_1<q\leq p\leq z_2 \\ z_3\leq pq<z_5 \\ z_6\leq q^2p\leq  X \\ q<r\leq (X/pq)^{1/2}}} S_{pqr}^*\left(r, \theta\right)\:.
\end{align*} 

The first term above is counting products of exactly three primes, and for these terms we drop the contribution of the $\mathcal{A}$-part for a non-negative sum. We obtain
\[
\sum_{\substack{z_1<q\leq p\leq z_2 \\ z_3< pq\leq z_5 \\ z_6\leq q^2p\leq  X}} S_{pq}^*\left(  \left(\frac{X}{pq}\right)^{1/2}, \theta\right)  = P_4(\theta)-\frac{\kappa_{\mathcal{A}}\#\mathcal{A}^*}{\log X} E_2(I_2, \theta)+N_8(\theta)\:,  \tag{4.7}
\]
where
$$I_2=\iint_{\substack{\theta_2-\theta_1<v<u<\theta_1 \\ \theta_2<u+v<1-\theta_2 \\ 1-\theta_1<2v+u<1  }} \frac{dudv}{uv(1-u-v)}\:.$$
For the terms not coming from products of three primes, we split our summation according to the size of $pqr$, noting that this is negligible, if 
$qr\in [z_2, z_3)$ by Proposition \ref{prop42}. For the terms with 
$qr\not\in [z_2, z_3)$ we just take the trivial lower bound. Thus, by Proposition \ref{prop43} we have

\begin{align*}
&\sum_{\substack{z_1<q\leq p\leq z_2 \\ z_3\leq  pq\leq z_5 \\ z_6\leq q^2p\leq  X \\ q<r\leq  \left(\frac{X}{pq}\right)^{1/2} }}  S_{pqr}^* (r,\theta)=\\
&=\sum_{\substack{z_1\leq q\leq p\leq z_2 \\ z_3\leq  pq< z_5 \\ z_6\leq q^2p\leq  X \\ q<r\leq  \left(\frac{X}{pq}\right)^{1/2} \\ qr<z_2 }}  S_{pqr}^* (r,\theta)
+
\sum_{\substack{z_1< q\leq p\leq z_2 \\ z_3\leq  pq< z_5 \\ z_6\leq q^2p\leq  X \\ q<r\leq  \left(\frac{X}{pq}\right)^{1/2} \\ qr>z_3 }}  S_{pqr}^* (r,\theta) + N_9(\theta)\\
&=-(1+o(1))\: \frac{\kappa_{\mathcal{A}}\#\mathcal{A}^*}{\log X} (E_3(I_3,\theta)+E_4(I_4,\theta))+P_5(\theta),
\end{align*}

where
\[
I_3=\iiint_{(u,v,w)\in\mathcal{R}_1} \omega\left( \frac{1-u-v-w}{w} \right)\:\frac{dudvdw}{uvw^2}\:,  \tag{4.8}
\]
\[
I_4=\iiint_{(u,v,w)\in\mathcal{R}_2} \omega\left( \frac{1-u-v-w}{w} \right)\:\frac{dudvdw}{uvw^2}\:,  \tag{4.9}
\]
where $\mathcal{R}_1$ and $\mathcal{R}_2$ are given by
$$\mathcal{R}_1:=\{(u,v,w)\::\: \theta_2-\theta_1<v<u<\theta_1, 
\theta_2<u+v<1-\theta_2,$$ 
$$1-\theta_1<u+2v<1, v<w<(1-u-v)/2, v+w<\theta_1\}\:,$$
$$\mathcal{R}_2:=\{(u,v,w)\::\: \theta_2-\theta_1<v<u<\theta_1, 
\theta_2<u+v<1-\theta_2,$$ 
$$1-\theta_1<u+2v<1, v<w<(1-u-v)/2, v+w<\theta_2\}\:.$$

When $q^2p<z_6$ we can apply two further Buchstab iterations, since then we can evaluate terms $S_{pqr}^*(z_1, \theta)$ with $r\leq q\leq p$ using Proposition \ref{prop41} (since $pqr<z_6$). This gives
\begin{align*}
&\sum_{\substack{z_1<q\leq p\leq z_2 \\ q^2p<z_6 \\ z_3\leq pq< z_5}} S_{pq}^*(q, \theta)= \sum_{\substack{z_1<q\leq p\leq z_2 \\ q^2p<z_6 \\ z_3\leq pq< z_5}} S_{pq}^*\left( \min\left(q,\left(\frac{X}{pq}\right)^{1/2}\right), \theta \right)+N_{10}(\theta)\\
&= \sum_{\substack{z_1<q\leq p\leq z_2 \\ q^2p<z_6 \\ z_3\leq pq< z_5}} S_{pq}^*(z_1, \theta) - \sum_{\substack{z_1<r\leq q\leq p\leq z_2 \\ q^2p<z_6 \\ z_3\leq pq< z_5 \\ r\leq (X/pq)^{1/2}}} S_{pqr}^*(r, \theta) +N_{11}(\theta)\\
&=N_{12}(\theta)- \sum_{\substack{z_1<r\leq q\leq p\leq z_2 \\ q^2p<z_6 \\ z_3\leq pq\leq  z_5 \\ r\leq (X/pq)^{1/2}}} S_{pqr}^*\left( \min\left(r,\left(\frac{X}{pqr}\right)^{1/2}\right), \theta \right)\\
&=N_{12}(\theta)-\sum_{\substack{z_1<r\leq q\leq p\leq z_2 \\ q^2p<z_6 \\ z_3\leq pq<  z_5 \\ r\leq (X/pq)^{1/2}}} S_{pqr}^*(z_1, \theta)+
\sum_{\substack{z_1<s\leq r\leq q\leq p\leq z_2 \\ q^2p<z_6 \\ z_3\leq pq<  z_5 \\ r^2pq, s^2rpq\leq X}} S_{pqrs}^*(s, \theta)\\
&=N_{13}(\theta)+\sum_{\substack{z_1<s\leq r\leq q\leq p\leq z_2 \\ q^2p\leq z_6 \\ z_3\leq pq<  z_5 \\ r^2pq, s^2rpq\leq X}} S_{pqrs}^*(s, \theta)\:,
\end{align*}
where $r,s$ are restricted to primes in the sums above.\\
Finally, we see that any part of the final sum with a product of two of $p,q,r,s$ in $[z_2, z_3]$ can be discarded by Proposition \ref{prop42}. Trivially lower bounding the remaining terms as before yields
$$\sum_{\substack{z_1<s\leq r\leq q\leq p\leq z_2 \\ q^2p<z_6 \\ z_3\leq pq<  z_5 \\ r^2pq, s^2rpq\leq X}} S_{pqrs}^*(s, \theta)=P_6(\theta)-\frac{\kappa_{\mathcal{A}}\#\mathcal{A}^*}{\log X}\: E_5(I_5, \theta)\:,$$
with 
\[
I_5=\iiiint_{(u,v,w,t)\in \mathcal{R}_3} \omega\left(\frac{1-u-v-w-t}{t}\right)\frac{dudvdwdt}{uvwt^2}\:, \tag{4.10}
\]
where $\mathcal{R}_3$ is given by
$$\mathcal{R}_3:=\{(u,v,w,t)\::\: \theta_2-\theta_1<t<w<v<u<\theta_1, 
u+2v<1-\theta_1,$$
$$ u+v+2w<1, u+v+w+2t<1, \theta_2<u+v<1-\theta_2,$$
$$\{ u+v, u+w, u+t, v+w, v+t, w+t\}\cap[\theta_1, \theta_2]=\emptyset \}. $$

We perform decompositions to the third term of (4.3) in a similar way to how we dealt with the second term. We have 
$$q^2p<(qp)^{3/2}<z_2^{3/2}<z_6$$
so, as above, we can apply two Buchstab iterations and use Proposition \ref{prop41} to deal with the terms $S_{pqr}(z_1, \theta)$, since we have $pqr\leq pq^2<z_6$.\\
Furthermore, we notice that terms with any of $pqr, pqs, prs,$ or $qrs$ in $[z_2, z_3]\cup [z_5, z_6]$ are negligible by Proposition \ref{prop42}. This gives 
\begin{align*}
\sum_{\substack{z_1<q\leq p\leq z_2 \\ z_1\leq pq\leq z_2}} S_{pq}^*(q, \theta)&=  \sum_{\substack{z_1<q\leq p\leq z_2 \\ z_1\leq pq< z_2}} S_{pq}^*(z_1, \theta)-\sum_{\substack{z_1<r\leq q\leq p\leq z_2 \\ z_1\leq pq< z_2}} S_{pq}^*(r, \theta)\\
&=N_{14}(\theta)-\sum_{\substack{z_1<r\leq q\leq p\leq z_2 \\ z_1\leq pq\leq z_2}} S_{pqr}^*(z_1, \theta)+\sum_{\substack{z_1<s<r<q< p< z_2 \\ z_1< pq< z_2}} S_{pqrs}^*(s, \theta)\\
&=\sum_{\substack{z_1<s<r<q< p< z_2 \\ z_1< pq< z_2 \\ prq, pqs, prs, qrs\not\in[z_2, z_3]\\ pqrs\not\in[z_2, z_3]\cup[z_5, z_6]}} S_{pqrs}^*(s, \theta)+N_{15}(\theta)\\
&=P_7(\theta)-\frac{\kappa_{\mathcal{A}}\#\mathcal{A}^*}{\log X} E_6(I_6, \theta)\:,
\end{align*}
where
\[
I_6=\iiiint_{(u,v,w,t)\in\mathcal{R}_4} \omega\left(\frac{1-u-v-w-t}{t}\right)\frac{dudvdwdt}{uvwt^2}\:,\tag{4.11}
\]
where
$$\mathcal{R}_4:=\{(u,v,w,t)\::\: \theta_2-\theta_1<t<w<v<u<\theta_1, u+v<\theta_1, $$
$$u+v+w+t\not\in[\theta_1, \theta_2]\cup [1-\theta_2, 1-\theta_1],$$
$$\{u+v+w, u+v+t, u+w+t, v+w+t\}\cap[\theta_1, \theta_2]=\emptyset\}.$$
Together (4.4), (4.6), (4.7), (4.8), (4.9), (4.10) and (4.11) give our lower bound for all the terms occurring in (4.3) and so give a lower bound for the first term from (4.3) which covers all terms with $p\leq z_2$.\\
We are left to consider the second term from (4.3), which is the remaining term with $p\in (z_3, z_4]$. We treat these in a similar manner to those with $p\leq z_2$.\\
We first split the sum according to the size of $qp$. Terms with $qp\in[z_5, z_6)$ are negligible by Proposition \ref{prop42}, so we are left to consider $qp\in(z_3, z_5)$ or $qp>z_6$. We then split the terms with $qp\in(z_3, z_5)$ according to the size of $q^2p$ compared with $z_6$. This gives
$$\sum_{\substack{z_3<p\leq z_4 \\ z_1<q\leq (X/p)^{1/2}}} S_{pq}^*(q, \theta)=S_1+S_2+S_3+N_{16}(\theta),$$
where
$$S_1:=\sum_{\substack{z_3<p\leq z_4 \\ z_1<q\leq (X/p)^{1/2}}} S_{pq}^*(q, \theta)=P_8(\theta)-\frac{\kappa_{\mathcal{A}}\#\mathcal{A}^*}{\log X}\: E_7(I_7, \theta)\:,$$
with
$$I_7=\iint_{\substack{\theta_2<u<1/2 \\ \theta_2-\theta_1<v<(1-u)/2 \\ 1-\theta_1<u+v}} \omega\left(\frac{1-u-v}{v}\right)\:\frac{dudv}{uv^2}$$
$$S_2:=\sum_{\substack{z_3<p\leq z_4 \\ z_1<q\leq (X/p)^{1/2} \\ z_3<qp\leq z_5 \\ z_6\leq q^2p}} S_{pq}^*(q, \theta)=P_9(\theta)-\frac{\kappa_{\mathcal{A}}\#\mathcal{A}^*}{\log X}\: E_8(I_8, \theta)\:,$$
with
\[
I_8=\iint_{\substack{\theta_2<u<1/2 \\ \theta_2-\theta_1<v<(1-u)/2 \\ \theta_2<u+v<1-\theta_2 \\ 1-\theta_1<2v+u}} \omega\left(\frac{1-u-v}{v}\right)\:\frac{dudv}{uv^2}  \tag{4.12}
\]
and where
$$S_3:=\sum_{\substack{z_3<p\leq z_4 \\ z_1<q\leq (X/p)^{1/2} \\ z_3<qp< z_5 \\ q^2p<z_6}} S_{pq}^*(q, \theta)\:.$$
We apply two further Buchstab iterations to $S_3$ (we can handle the intermediate terms using Proposition \ref{prop41} as before since $q^2p<z_6$).\\
As before, we may replace $S_{pqr}^*(q, \theta)$ by 
$$S_{pq}^*\left(\min\left(q, \left(\frac{X}{pq}\right)^{1/2}, \theta\right)\right)$$
and $S_{pqr}^*(r, \theta)$ by
$$S_{pqr}^*\left(\min\left(r, \left(\frac{X}{pqr}\right)^{1/2}, \theta\right)\right)$$
at the cost of a negligible error term (since $pqr<z_6$). This gives

\begin{align*}
S_3&:=\sum_{\substack{z_3<p\leq z_4 \\ z_1<q\leq (X/p)^{1/2} \\ z_3\leq qp< z_5 \\ q^2p<z_6}} S_{pq}^*\left(\min\left(q, \left(\frac{X}{pq}\right)^{1/2}, \theta\right)\right)+N_{17}(\theta)\\
&=\sum_{\substack{z_3<p\leq z_4 \\ z_1<q\leq (X/p)^{1/2} \\ q^2p<z_6 \\ z_3<qp<z_5}} S_{pq}^*(z_1, \theta)\\
&\ \ \ -\sum_{\substack{z_3<p\leq z_4 \\ z_1<r<q\leq (X/p)^{1/2}}}S_{pqr}^*\left(\min\left(r, \left(\frac{X}{pqr}\right)^{1/2}, \theta\right)\right)+N_{18}(\theta)\\
&=N_{18}(\theta)-\sum_{\substack{z_3<p\leq z_4 \\ z_1<r\leq q\leq (X/p)^{1/2}\\  q^2p<z_6 \\ z_3<qp<z_5 \\ r^2pq\leq X}} S_{qpr}^*(z_1, \theta)+ 
\sum_{\substack{z_3<p\leq z_4 \\ z_1<s\leq r\leq q\leq (X/p)^{1/2}\\  q^2p<z_6 \\ z_3<qp<z_5 \\ s^2qrp, r^2pq\leq X}} S_{qprs}^*(s, \theta)\\
&=N_{19}(\theta)-\frac{\kappa_{\mathcal{A}}\#\mathcal{A}^*}{\log X}\: E_9(I_9, \theta)\:,
\end{align*}
where 
\[
I_9=\iiiint_{(u,v,w,t)\in\mathcal{R}_5} \omega\left(\frac{1-u-v-w-t}{t}\right)\:\frac{dudvdwdt}{uvwt^2}\:,  \tag{4.13}
\]
with
$$\mathcal{R}_5:=\{(u,v,w,t)\::\:\theta_2-\theta_1<t<w<v, \theta_2<u<1/2, u+2v<1-\theta_1,$$
$$u+v+2w<1, u+v+w+2t<1, \theta_2<u+v<1-\theta_2,$$
$$\{ u+v, u+w, u+t, v+w, v+t, w+t \}\not\in [\theta_1, \theta_2]\}\:.$$
From (4.5), (4.7), (4.8), (4.9), (4.10), (4.11), (4.12), (4.13) we now obtain
\begin{align*}
J(S_{\mathcal{A}^*\cap \mathcal{P}}(\theta))=&  \frac{X}{4}\frac{\#\mathcal{A}^*}{\log X} \mathfrak{S}(N_0)\\
&\ \ \times (1-I_1-I_2-I_3-I_4-I_5-I_6-I_7-I_8-I_9)\:.   \tag{4.14}
\end{align*}
Numerical integration (Maynard \cite{maynard} has included a Mathematica${}^{\textregistered}$ file detailing this computation with this article on arxiv.org) shows that 
\[
I_1+\cdots+I_9<0.996. \tag{4.15}
\]
Thus Theorem \ref{thm11} follows from (4.14) and (4.15). It remains to prove Propositions \ref{prop41}, \ref{prop42} and \ref{prop43}.

\section{Fourier estimates and Large Sieve inequalities}

In this section we collect various results related to Section 10 of Maynard. We also cite the Large Sieve inequality from Analytic Number Theory.

\begin{definition}\label{def51}
Let
$$\mathcal{A}_1:=\left\{\sum_{0\leq i\leq k} n_i10^i\::\: n_i\in\{0, \ldots, 9\}\setminus\{a_0\}, k\geq 0  \right\}\:.$$
For $Y$ an integral power of 10, we write
$$F_Y(\theta):=Y^{-\log 9/\log 10} \left| \sum_{n<Y} 1_{\mathcal{A}_1}(n) e(n\theta)  \right|\:.$$
\end{definition}

\begin{lemma}\label{lem51}
Let $q<Y^{1/3}$ be of the form $q=q_1q_2$ with $(q_1, 10)=1$ and $q_1>1$, and let $|\eta|<Y^{-2/3}/2$. Then for any integer $c$ coprime to $q$ we have
$$F_Y\left(\frac{c}{q}+\eta\right)\ll \exp\left(-c_1\:\frac{\log Y}{\log q}\right)$$
for some absolute constant $c_1>0$.
\end{lemma}
\begin{proof}
This is Lemma 10.1 of \cite{maynard}.
\end{proof}

\begin{lemma}\label{lem52}
We have for $Y_1\asymp Y_2 \asymp Y_3$
$$\sup_{\beta\in\mathbb{R}}\sum_{c< Y_1} F_{Y_2}\left(\beta+\frac{c}{Y_3}\right)\ll Y_1^{27/77}$$
and
$$\int_0^1 F_Y(t)dt\ll \frac{1}{Y^{50/77}}\:.$$
\end{lemma}
\begin{proof}
This is contained in Lemma 10.2 of \cite{maynard}.
\end{proof}

\begin{lemma}\label{lem53}
We have that
$$\#\left\{0\leq c< Y\::\: F_{Y}\left(\frac{c}{Y}\right)\sim \frac{1}{B}  \right\} \ll B^{235/154}Y^{59/433}\:.$$
\end{lemma}
\begin{proof}
This is Lemma 10.4 of \cite{maynard}.
\end{proof}

\begin{lemma}\label{lem54}(Large sieve estimates)\\
We have
$$\sup_{\beta\in\mathbb{R}} \sum_{c\leq q} \sup_{|\eta|<\delta} F_Y\left(\frac{c}{q}+\beta+\eta\right)\ll (1+\delta q)\left(q^{27/77}+\frac{q}{Y^{50/77}}\right)$$
$$\sup_{\beta\in\mathbb{R}} \sum_{q\leq Q}\sum_{\substack{0<c<q \\ (c,q)=1}} \sup_{|\eta|<\delta} F_Y\left(\frac{c}{q}+\beta+\eta\right)\ll (1+\delta Q^2)\left(Q^{54/77}+\frac{Q^2}{Y^{50/77}}\right)$$
and for any integer $d$, we have
$$\sup_{\beta\in\mathbb{R}} \sum_{\substack{q\leq Q \\ d\mid q}}\sum_{\substack{0<c<q \\ (c,q)=1}} \sup_{|\eta|<\delta} F_Y\left(\frac{c}{q}+\beta+\eta\right)\ll \left(1+\frac{\delta Q^2}{d}\right)\left(\left(\frac{Q^{2}}{d}\right)^{27/77}+\frac{Q^2}{dY^{50/77}}\right)$$
\end{lemma}
\begin{proof}
This is Lemma 10.5 of \cite{maynard}.
\end{proof}


\begin{lemma}\label{lem55}(Hybrid Bounds)\\
Let $E\geq 1$. Then we have
$$\sum_{c\leq q}\sum_{\substack{|\eta|\leq E/Y \\ \left(\eta+\frac{c}{q}\right)Y\in\mathbb{Z}}} F_Y\left(\frac{c}{q}+\eta\right)\ll (qE)^{27/77}+\frac{qE}{Y^{50/77}}\:,$$
$$\sum_{\substack{q<Q \\ d\mid q}}\sum_{\substack{c\leq q \\ (c,q)=1}}\sum_{\substack{|\eta|\leq E/Y \\ \left(\eta+\frac{a}{q}\right)Y\in\mathbb{Z}}} F_Y\left(\frac{c}{q}+\eta\right)\ll \left(\frac{Q^2E}{d}\right)^{27/77}+\frac{Q^2E}{dY^{50/77}}\:.$$
\end{lemma}

\begin{proof}
This is Lemma 10.6 of \cite{maynard}.
\end{proof}

\begin{lemma}\label{lem56} (Alternative Hybrid Bound)\\
Let $D, E, Y, Q_1\geq 1$ be integral powers of 10 with $DE\ll Y$. Let $q_1\sim Q_1$ with $(q_1, 10)=1$ and let $d\sim D$ satisfy $d|10^u$ for some $u\geq 0$. Let
\begin{align*}
S&=S(d, q_1, Q_2, E, Y)\\
&:=\sum_{\substack{q_2\sim Q_2 \\ (q_2, 10)=1}}\sum_{\substack{c<dq_1q_2 \\ (c, dq_1q_2)=1}}\sum_{\substack{|\eta|\leq E/Y \\ \left(\eta+\frac{c}{q_1q_2d}\right)Y\in\mathbb{Z}}}
 F_Y\left(\frac{c}{dq_1q_2} +\eta\right)\:.
\end{align*}
Then we have 
$$S\ll (DE)^{27/77}(Q_1Q_2^2)^{1/21}+\frac{E^{5/6}D^{3/2}Q_1Q_2^2}{Y^{10/21}}\:.$$
In particular, if $q=dq'$ with $(q',10)=1$ and $d\mid 10^u$ for some integer $u\geq 0$, then we have
$$\sum_{\substack{c<q \\ (c, q)=1}}\sum_{\substack{|\eta|\leq E/Y \\ \left(\eta+\frac{c}{q}\right)Y\in\mathbb{Z}}}
 F_Y\left(\frac{c}{q} +\eta\right)\ll (dE)^{27/77} q^{1/21}+\frac{E^{5/6}d^{3/2}q}{Y^{10/21}}\:.$$
\end{lemma}

\begin{proof}
This is Lemma 10.7 of \cite{maynard}.
\end{proof}

\begin{lemma}\label{lem57}(Large Sieve Estimates)\\
Let $\alpha_r \in\mathbb{R}\setminus\mathbb{Z}$, $\|\alpha_i-\alpha_j\|\geq \delta$, $a_n$ complex numbers with $M<n\leq M+N$, where $0<\delta<1/2$ and $N\geq 1$ is an integer. Then 
$$\sum_r\left| \sum_{M<n\leq M+N} a_n e(\alpha_r n) \right|^2 \leq (\delta^{-1}+N-1)\|a\|^2  \:.$$
\end{lemma}
\begin{proof}
This is Theorem 7.11 of \cite{IKW}
\end{proof}


\section{Local versions of Maynard's results}

In this section we reduce the propositions of Section 4 to other facts, which will be proven in later sections.\\
Whereas we may represent the sifted sets appearing in Propositions \ref{prop42} and  \ref{prop43} as a union of simpler sets, the set considered in Proposition \ref{prop41} is obtained by an idea related to the inclusion-exclusion principle. Its analogue in number theory in its simplest form is the Sieve of Eratosthenes-Legendre containing the M\"obius $\mu$-function $\mu(n)$:\\
Let $\mathcal{C}$ be a set of integers and $\mathcal{P}$ a set of primes then 
\begin{align*}
S(\mathcal{C}, \mathcal{P}, z)&:=\#\{n\in \mathcal{C}\::\: p\mid n,\ p\in \mathcal{P}\ \Rightarrow \ p>z\}\\
&=\sum_{n\in \mathcal{C}}\sum_{\substack{t\mid n \\ t\mid P(z)}}\mu (t)\:,\ \ \text{with}\ P(z):=\prod_{\substack{p\leq z \\ p\in \mathcal{P}}} p\:.
\end{align*}

In the theory of combinatorial sieves the M\"obius function is replaced by a function $\lambda$, having smaller support. Also in this paper we proceed in this way. The basis is the following result from Combinatorial Sieve Theory.

\begin{lemma}\label{rlem61}
Let $\kappa>0$ and $y>1$. There exist two sets of real numbers 
$$\Lambda^+=(\lambda^+_d)\ \ \text{and}\ \ \Lambda^-=(\lambda^-_d)$$
depending only on $\kappa$ and $y$ with the following properties:
\[
\lambda_1^{\pm}=1 \tag{6.1}
\]
\[
|\lambda_d^{\pm}| \leq 1\:,\ \ \text{if}\ 1\leq d< y \tag{6.2}
\]
\[
\lambda_d^{\pm}=0\:,\ \ \text{if}\ d\geq  y \tag{6.2}
\]
and for any integer $n>1$,
\[ 
\sum_{d\mid n} \lambda_d^-  \leq 0  \leq \sum_{d\mid n} \lambda_d^+\:.  \tag{6.3}
\]
Moreover, for any multiplicative function $g(d)$ with $0\leq g(p)<1$ and satisfying the dimension conditions 
$$\prod_{w\leq p< z} (1-g(p))^{-1}\leq \left(\frac{\log z}{\log w}\right)^\kappa \left(1+\frac{\kappa}{\log w}\right)$$
for all $2\leq w< y$, we have
$$\sum_{d \mid P(z)} \lambda_d^\pm g(d)=\left(1+O\left(e^{-s}\left(1+\frac{\kappa}{\log z}\right)^{10}\right)\right)\prod_{p< z} (1-g(p))\:,$$
where $P(z)$ denotes the product of all primes $p<z$ and $s=\log y/\log z$. The implied constants depend only on $\kappa$.
\end{lemma}
\begin{proof}
This is the Fundamental Lemma 6.3 of \cite{IKW}.
\end{proof}

A special role in the consideration of the sifted set in Proposition \ref{prop41} is played by the prime factors of 10, $p=2$ and $5$. To simplify things we consider the subsets
$$\mathcal{U}^{*'}:=\{a\in \mathcal{U}^*\::\: (a, 10)=1\}$$
and
$$\mathcal{B}^{*'}:=\{ b\in \mathcal{B}^*\::\: (b, 10)=1\}$$

\begin{definition}\label{def62}
Let $\lambda$ be an arithmetic function, $z\geq 1$. For a set of $\mathcal{C}$ of integers we define 
$$S(\mathcal{C}, z, \theta, \lambda):=\sum_{n\in \mathcal{C}} e(n\theta)\left(\sum_{\substack{t\mid n \\ t\mid P(z)}} \lambda(t)\right)$$
\end{definition}
We are now ready for the statement of 

\begin{proposition}\label{prop63}
Let $\epsilon>0$, $0<\eta_0\leq \theta_2-\theta_1$, $l=l(\eta_0)$ be fixed. Let $\mathcal{L}$, and the summation condition $\sum^\sim$ be as in Proposition \ref{prop41}, $q\leq Q_0$, $(c, q)=1$. Let $\lambda^\pm$ satisfy the properties of Lemma \ref{rlem61} with $y=X^{(\eta_0^{1/2})}$ and let $\lambda^\pm(t)=0$, if $(t, 10)>1$. Then we have for $\lambda=\lambda^-$ or $\lambda^+$:
$$\sum_{X^{\eta_0}\leq p_1\leq \cdots \leq p_l}^\sim \left( S\left(\mathcal{A}^*_{p_1 \ldots p_l}, X^{\eta_0}, \frac{c}{q}, \lambda\right)-
\frac{\kappa_{\mathcal{A}}\#\mathcal{A}^*}{\#\mathcal{B}^*}  S\left(\mathcal{B}^*_{p_1 \ldots p_l}, X^{\eta_0}, \frac{c}{q}, \lambda\right)  \right)$$
$$=O\left((\#\mathcal{A}^*)(\log X)^{-A}\right)\:.$$
\end{proposition}

\begin{proposition}\label{prop64}(Type II terms, local version)\\
Let $\epsilon, \eta_0, l, \mathcal{L}, \sum^\sim, q, c, t$ be as in Proposition \ref{prop63}. Then we have
$$\sum^\sim_{\substack{X^{\eta_0}\leq p_1\leq \cdots \leq p_l \\ X^{\theta_1}\leq \prod_{i\in \mathcal{I}} p_i\leq X^{\theta_2}\\ p_1\cdots p_l\leq X/p_j}} \left( S\left(\mathcal{A}^*_{p_1 \ldots p_l}, p_j, \frac{c}{q}\right)-\frac{\kappa_{\mathcal{A}}\#\mathcal{A}^*}{\#\mathcal{B}^*}  S\left(\mathcal{B}^*_{p_1 \ldots p_l}, p_j, \frac{c}{q}\right)  \right) $$
$$=O\left((\#\mathcal{A}^*)(\log X)^{-A}\right)\:.$$
and
$$\sum^\sim_{\substack{X^{\eta_0}\leq p_1\leq \cdots \leq p_l \\ X^{1-\theta_2}\leq \prod_{i\in \mathcal{I}} p_i\leq X^{1-\theta_1}\\ p_1\cdots p_l\leq X/p_j}} \left( S\left(\mathcal{A}^*_{p_1 \ldots p_l}, p_j, \frac{c}{q}\right)-\frac{\kappa_{\mathcal{A}}\#\mathcal{A}^*}{\#\mathcal{B}^*}  S\left(\mathcal{B}^*_{p_1 \ldots p_l}, p_j, \frac{c}{q}\right)  \right) $$
$$=O\left((\#\mathcal{A}^*)(\log X)^{-A}\right)\:.$$
\end{proposition}

We have now collected the material for the $a$-variable major arcs part of the local version of Maynard. For the deduction of Propositions \ref{prop41} and \ref{prop42} from Propositions \ref{prop63} and \ref{prop64} we also need information  on the Piatetski-Shapiro sum $S_{c_0}(\theta)$ as well as  information on the contributions of the  minor arcs.

\begin{lemma}\label{lem61}
Assume  the GRH. Let $c, q$ be positive integers, $(c,q)=1$. Then we have for all $\epsilon>0$:
$$\sum_{p\leq N} e\left(\left(\frac{c}{q}+\xi\right)p\right)\log p=\frac{\mu(q)}{\phi(q)} \sum_{n\leq N} e(n\xi)+O_\epsilon(N^{\frac{3}{2}+\epsilon} q(|\xi|)).\ \ (N\rightarrow \infty).$$
\end{lemma}

\begin{proof}
We decompose the sum into partial sums:
$$\sum_{p\leq N} e\left(\left(\frac{c}{q}+\xi\right)p\right)\log p= \sum_{s\bmod q} e\left(s\:\frac{c}{q}\right)\:\sum_{\substack{p\leq N \\ p\equiv s \bmod q}} e(\xi p)\log p\:.$$
We have 
$$  \sum_{\substack{p\leq N \\ p\equiv s \bmod q}} e(\xi p)\log p=\sum_{n\leq N} e(n\xi)\left(\frac{1}{\phi(q)}+r(n)\right)\:.$$
For the evaluation of the sum 
$$\sum_{n\leq N} r(n)e(n\xi)$$
we apply summation by parts and use the well known consequence of the GRH:
$$\sum_{\substack{n\equiv s \bmod q \\ n\leq u  }} \Lambda(n)=\frac{u}{\phi(q)}+O_\epsilon(u^{\frac{1}{2}+\epsilon})\ \ (\epsilon>0).$$
\end{proof}

\begin{lemma}\label{lem62}
Let $\gamma, \delta$ satisfy $0<\gamma\leq 1$, $0<\delta$ and 
$$9(1-\gamma)+12\delta<1.$$
Then, uniformly in $\alpha$, we have
$$\frac{1}{\gamma}\sum_{\substack{p<N \\ p=[n^{\frac{1}{\gamma}}]  }} e(\alpha p) p^{1-\gamma}\log p=\sum_{p<N}e(\alpha p)\log p+O(N^{1-\delta})\:,$$
where the implied constant may depend on $\gamma$ and $\delta$ only.
\end{lemma}

\begin{proof}
This is Theorem 4 of \cite{balog}.
\end{proof}

We now discuss the minor arcs contributions to the expression (2.6) for $\mathcal{J}(E)$, where $E$ is the exponential sum $E_0$ from Proposition \ref{prop41} or one of the exponential sums $E_1$ or $E_2$ of Proposition \ref{prop42}.

\begin{lemma}\label{lem63}
Suppose that $\theta\in\mathbb{R}$ such that there are integers $c,q$ with $(c,q)=1$ and 
$$\left|\theta-\frac{c}{q}\right|< q^{-2}.$$
Then, for all $N\geq 2$ we have
$$\sum_{n\leq N} \Lambda(n)e(\theta n)\ll (\log N)^{\frac{7}{2}}\left(\frac{N}{q^{\frac{1}{2}}}+N^{\frac{4}{5}}+(Nq)^{\frac{1}{2}}\right).$$
\end{lemma}

\begin{proof}
This is Theorem 2.1 of  \cite{maynard}.
\end{proof}

 \begin{lemma}\label{lem615}
Given $\theta\in (0,1)$, $N\in\mathbb{N}$. Then there is $q$ with $1\leq q\leq N$, such that 
\[
\left|\theta-\frac{c}{q}\right|\leq \frac{1}{qN}.  \tag{6.15}
\]
\end{lemma}



\begin{lemma}\label{lem617}
The intervals $I_{c,q}(L)$, $1\leq q\leq [X^{4/5}]+1$ cover $(0,1)$.
\end{lemma}
\begin{proof}
This follows by application of Lemma $\ref{lem615}$ with $N=[X^{4/5}]+1$.
\end{proof}

\begin{definition}\label{def618}
Let $\delta_0>0$ be fixed, such that 
$$9(1-\gamma_0)+12\delta_0<1.$$
Then we define
$$\mathfrak{n}:=\{ \theta\in(0,1)\::\: \left|S_{c_0}(\theta)\right|\leq X^{1-\delta_0}\}.$$
\end{definition}


\begin{lemma}\label{lem619}
Let $E(\theta)$ be one of the exponential sums $E_i(\theta)$, $(i=0,1,2)$, considered in Propositions \ref{prop41}, \ref{prop42}. Then we have:
\[
\sum_{\substack{1\leq q\leq Q_0 \\ (c, q)=1}}\sum_{\frac{a}{X}\in I_{c,q}\setminus I_{c,q}(L_0)} E\left(\frac{a}{X}\right) S_{c_0}^2\left(\frac{a}{X}\right) e\left(-N_0\frac{a}{X}\right) \ll |\mathcal{A}| X^2 L_0^{-1}.  \tag{6.16}
\]
\end{lemma}
\begin{proof}
By Lemma \ref{lem62} and GRH we have for 
$$\xi\in \bigcup_{\substack{1\leq q \leq (\log X)^{c_1} \\ (c,q)=1}} I_{c,q} \setminus I_{c,q} (L_0)\:,$$
$$Q<q\leq 2Q\leq (\log X)^{c_1}:$$
\[
S_{c_0}\left(\frac{c}{q}+\xi\right)=\frac{\mu(q)}{\phi(q)}\sum_{m\in Int} e(m\xi)+O(X^{1-\delta_0}) \ll \frac{X}{QL}  \tag{6.17}
\]
\[
\#\left\{a\::\: \frac{a}{X}\in I_{c,q}(L),\ Q<q\leq 2Q\right\}\leq QL  \tag{6.18}
\]
and
$$E\left(\frac{a}{X}\right)\ll |\mathcal{A}|.$$
From (6.17) and (6.18) we obtain 
\[ 
\sum_{\substack{1\leq q \leq (\log X)^{C_1} \\ (c,q)=1}} \ \sum_{\frac{a}{X}\in I_{c,q}(2L) \setminus I_{c,q}(L)} E\left(\frac{a}{X}\right)S_{c_0}^2 \left(\frac{a}{X}\right) e\left(-N_0\:\frac{a}{X}\right) \ll |\mathcal{A}| X^2 L^{-1}. \tag{6.19}
\]
Summation of (6.9) for $L=L_0 2^j$ gives the result of Lemma \ref{lem619}.
\end{proof}

\begin{lemma}\label{lem620}
$$\sum^{\sim}:=\sum_{\frac{a}{X}\in n} E\left(\frac{a}{X}\right) S_{c_0}^2 \left(\frac{a}{X}\right) e\left(-N_0\frac{a}{X}\right) \ll |\mathcal{A}|^{1/2} X^{5/2-\delta_0}.$$
\end{lemma}
\begin{proof}
We apply the Caychy-Schwarz inequality and Parseval's equation. Observing the definition of $n$ we get:
$$\sum^{\sim} \leq \left( \sum_{1\leq a\leq X} E \left(\frac{a}{X}\right)^2 \right)^{1/2} X^{1-\delta_0} \left(\sum_{1\leq a\leq X} S_{c_0}^2  \left(\frac{a}{X}\right)\right)^{1/2} \ll X^{1/2} |\mathcal{A}|^{1/2} X^{1-\delta_0} X,$$
i.e. Lemma \ref{lem620}.
\end{proof}

\begin{lemma}\label{lem621}
Let $Q\leq X^{\delta_0}$, $L\leq X^{1/5}$. Then we have
$$\sum_{Q< q\leq 2Q}\  \sum_{c \bmod q}\  \sum_{\frac{a}{X}\in I_{c,q}(L)} E \left(\frac{a}{X}\right) S_{c_0}^2  \left(\frac{a}{X}\right) e(-N_0\frac{a}{X}) \ll |\mathcal{A}| X^2 Q^{-1/2}.$$
\end{lemma}
\begin{proof}
We only deal with the $\mathcal{A}$-part of $E(\theta)$. We partition the intervals $I_{c,q}$ into subintervals:
$$I_{c,q}(L)=\bigcup_{j=-R_{c,q,L}}^{R_{c,q,L}}\mathcal{H}_j  $$
with 
$$\mathcal{H}_j=\left(\frac{c}{q}+q^{-1} X^{-1} 2^j,\ \frac{c}{q}+q^{-1} X^{-1} 2^{j+1}\right)\:,\ (j\geq 0),$$
$$R_{c, q, L}\ll \log L$$
(and analogous definition for $j<0$).\\
We write $\xi_j=q^{-1} X^{-1} 2^j$. Let 
$$z_l=\frac{a_l}{X}\in \mathcal{H}_j,\ z_l=\frac{c}{q}+\xi_l.$$
We have 
$$E(z_l)=E\left(\frac{c}{q}+\xi_j\right)+\int_{\xi_j}^{\xi_l} E'\left(\frac{c}{q}+\xi\right)\: d\xi$$
and thus
$$|E(z_l)|\leq \left| E\left(\frac{c}{q}+\xi_j\right)\right|+\int_{q^{-1}2^j X^{-1}}^{q^{-1}2^{j+1}X^{-1}} \left|E'\left(\frac{c}{q}+\xi\right)\right|\: d\xi.$$
From Lemma \ref{lem61} we obtain:
$$S_{c_0}\left(\frac{c}{q}+\xi\right)\ll \frac{1}{\phi(q)} \min(X, |\xi|^{-1}).$$
Thus we obtain
\begin{align*}
&\sum_{z_l\in \mathcal{U}_j} |E(z_l)|\: |S_{c_0}(z_l)|^2\\
&\ll \frac{1}{\phi(q)^2} \min (X^2, |\xi|^{-2})(1+q^{-1}2^{j+1})\\
&\ \times \left( \left|E\left(\frac{c}{q}+\xi_j\right)\right|+\int_{q^{-1}2^jX^{-1}}^{q^{-1}2^{j+1}X^{-1}}\left|E'\left(\frac{c}{q}+\xi\right)  \right| d\xi  \right).
\end{align*}
We now sum over all $c\bmod q$ and obtain
\begin{align*}
&\sum_{c\bmod q}\ \sum_{\frac{a}{X}\in I_{c,q}(L)} E\left(\frac{a}{X}\right) S_{c_0}^2\left(\frac{a}{X}\right) e\left(-N_0\frac{a}{X}\right)  \tag{6.20}\\
&\ll \frac{X^2}{\phi(q)^2} \sum_{j\::\: 2^j\leq L} \min(1, q2^{-j})(1+q^{-1}2^j)\\
&\ \ \sum_{c\bmod q}\left(\left| E\left(\frac{c}{q}+\xi_j\right)\right|+\int_{q^{-1}2^jX^{-1}}^{q^{-1}2^{j+1}X^{-1}}\left| E'\left(\frac{c}{q}+\xi\right)\right|\: d\xi \right).
\end{align*}
Recalling the definition:
\[
E(\theta)=\sum v(n) e(n\theta)\:,\ \text{with $v(n)=0$ for $n\not\in \mathcal{A}$} \tag{2.7}
\]
we have
\begin{align*}
&\sum_{c\bmod q} \left| E\left(\frac{c}{q}+\xi_j\right)\right|^2 \tag{6.21}\\ 
&=\sum_{n_1, n_2} v(n_1) v(n_2) e(n_1\xi_j)e(-n_2\xi_j)\sum_{c\bmod q} e\left(\frac{c(n_1-n_2)}{q}\right)\\
&\leq q\sum_{s \bmod q}\ \sum_{\substack{n_1\equiv n_2 \equiv s \bmod q \\ n_1, n_2\in\mathcal{A}}} 1 \ll |\mathcal{A}|^2
\end{align*}
by Lemma \ref{lem55}.\\
We apply an analogous argument to obtain
\[
\sum_{c\bmod q}\left| E'\left(\frac{c}{q}+\xi\right) \right|^2 \ll X^2 |\mathcal{A}|^2.  \tag{6.22}
\]
From (6.20), (6.21) and (6.22) we obtain Lemma \ref{lem621}.
\end{proof}

\begin{lemma}\label{lem622}
Let $X^{1/5}\leq q\leq [X^{4/5}]+1$. Then the interval $I_{c,q}(L)$ is a subset of $\mathfrak{n}$.
\end{lemma}
\begin{proof}
This follows from Lemma \ref{lem63}. Lemmas \ref{lem617}, \ref{lem619}, \ref{lem620},\ref{lem621} now show that the contributions to $J(E)$ in the formula (2.8) give a negligible contribution, if $C_1$ is chosen sufficiently large and $\eta$ sufficiently small.
\end{proof}

We now come to the proof of Proposition \ref{prop43}. We first prove a modification containing the weights $\lambda^\pm$.

\begin{lemma}\label{rlem616}
Let $\lambda^\pm$ satisfy the properties of Lemma \ref{rlem61} and $\lambda^\pm(t)=0$, if $(t, 10)>1$. Let
$$E_{0, \mathcal{A}^*, \lambda}(\theta)=\sum_{x^{\eta_0}\leq p_1\leq \cdots \leq p_l}^\sim S(\mathcal{A}_{p_1\cdots p_l}^*, X^{\eta_0}, \theta, \lambda)$$
$$E_{0, \mathcal{B}^*, \lambda}(\theta)=\sum_{x^{\eta_0}\leq p_1\leq \cdots \leq p_l}^\sim S(\mathcal{B}_{p_1\cdots p_l}^*, X^{\eta_0}, \theta, \lambda)$$
Then for $\lambda=\lambda^-$ or $\lambda^+$, we have:
$$\frac{1}{X}\sum_{1\leq a\leq X} \left( E_{0, \mathcal{A}^*, \lambda}\left(\frac{a}{X}\right)-\kappa_{\mathcal{A}}\frac{\#\mathcal{A}^*}{\#\mathcal{B}^*} E_{0, \mathcal{B}^*, \lambda}\left(\frac{a}{X}\right)\right)$$
$$S_{c_0}^2\left(\frac{a}{X}\right)e\left(-N_0 \frac{a}{X}\right)=O\left(\# \mathcal{A}^* X(\log X)^{-A}  \right).$$
\end{lemma}

\noindent\textit{Proof of Lemma \ref{rlem616} assuming Proposition \ref{prop63} }

The first step in the approximation of $E_{0, \mathcal{A}^*, \lambda}(\theta)$ inside $I_{c,q}$ consists in the replacement of the variable factors $e\left( n\left(\frac{c}{q}+\xi\right)\right)$ by $e\left(n_0\frac{c}{q}\right)e(n\xi)$, where $n_0$ is the midpoint of the interval $\mathcal{B}^*$. We set
$$\tilde{\lambda}(u)=\sum_{t\mid n} \lambda (t)$$
and obtain:
\begin{align*}
E_{0, \mathcal{A}^*, \lambda}\left(\frac{c}{q}+\xi\right)& =\sum_{X^{\eta_0}\leq p_1\leq \cdots\leq p_l}^\sim S\left(\mathcal{A}^*_{p_1\ldots p_l}, X^{\eta_0}, \left(\frac{c}{q}+\xi\right), \lambda\right)   \tag{6.23}\\
&=\sum_{X^{\eta_0}\leq p_1\leq \cdots\leq p_l}^\sim e(n_0\xi) \sum_{n\in\mathcal{B}^*}  1_{\mathcal{A}^*_{p_1\ldots p_l}}(n) e\left(n\frac{c}{q}\right) \tilde{\lambda}(n)\\
&\ \ +\sum_{X^{\eta_0}\leq p_1\leq \cdots\leq p_l}^\sim e(n_0\xi) \sum_{n\in\mathcal{B}^*}  1_{\mathcal{A}^*_{p_1\ldots p_l}}(n) e\left(n\frac{c}{q}\right) (e(n\xi-e(n_0\xi))\lambda_0(n)\\
&=:  E_{0, \mathcal{A}^*, \lambda}^{(1)}+ E_{0, \mathcal{A}^*, \lambda}^{(2)}
\end{align*}

We have

\[
E_{0, \mathcal{A}^*, \lambda}^{(1)} = e(n_0\xi)\sum_{X^{\eta_0}\leq p_1\leq \cdots \leq p_l}^\sim  S\left(\mathcal{A}^*_{p_1\ldots p_l}, X^{\eta_0}, \frac{c}{q}, \lambda\right)    \tag{6.24}
\]
and
\[
E_{0, \mathcal{A}^*, \lambda}^{(2)} = O\left(\sum_{X^{\eta_0}\leq p_1\leq \cdots \leq p_l}^\sim \ \sum_{n\in\mathcal{B}^*} 1_{\mathcal{A}^*_{p_1\ldots p_l}}(n)\: |n-n_0|\:|\xi| \:\tilde{\lambda_2}(n)\right) \:.     \tag{6.25}
\]

We obtain an analogous decomposition for the $\mathcal{B}^*$-part:
$$E_{0, \mathcal{B}^*, \lambda}:=\Sigma^{(1)}_{0,\mathcal{B}^*, \lambda }+E_{0, \mathcal{B}^*, \lambda}^{(2)} \:.$$

From (6.23), (6. 24), and (6.25) we obtain:
\[ 
\left| E_{0, \mathcal{A}^*, \lambda}\left(\frac{c}{q}+\lambda\right)- \kappa_{\mathcal{A}}\frac{\#\mathcal{A}^*}{\#\mathcal{B}^*} E_{0, \mathcal{B}^*, \lambda}\left(\frac{c}{q}+\lambda\right)\right|
 \tag{6.26}
\]
$$\leq \left| E_{0, \mathcal{A}^*, \lambda}\left(\frac{c}{q}\right)-\kappa_{\mathcal{A}}\frac{\#\mathcal{A}^*}{\#\mathcal{B}^*} E_{0, \mathcal{B}^*, \lambda}\left(\frac{c}{q}\right)\right|+|E_{0, \mathcal{A}, \lambda}^{(1)}|+\kappa_{\mathcal{A}}\frac{\#\mathcal{A}^*}{\#\mathcal{B}^*} |E_{0, \mathcal{A}, \lambda}^{(2)}|.$$
By Proposition \ref{prop63} we have:
\[
 \left| E_{0, \mathcal{A}^*, \lambda}\left(\frac{c}{q}\right)-\kappa_{\mathcal{A}}\frac{\#\mathcal{A}^*}{\#\mathcal{B}^*} E_{0, \mathcal{B}^*, \lambda}\left(\frac{c}{q}\right)\right|=O\left( \#\mathcal{A}(\log X)^{-A}\right)\:.  \tag{6.27}
\]
We also have:
\[
e(n\xi)-e(n_0\xi)=O(|n-n_0|\:|\xi|)\:.  \tag{6.28}
\]
Additionally, we have
\[
\sum_{X^{\eta_0}\leq p_1\leq \cdots\leq p_l}^{\sim}\: \sum_{n\in \mathcal{A}^*_{p_1\cdots p_l}} \sum_{t\mid n} \lambda(t)\ll \sum_{\substack{X^{\eta_0}\leq p_1\leq \cdots\leq p_l \\ t\leq X^6}}\ \sum_{n\in \mathcal{A}^*_{[p_1\cdots p_l, t]}} 1 \tag{6.29}
\]

\noindent We also have
\[
\sum_{X^{\eta_0}\leq p_1\leq \cdots\leq p_l}^{\sim}\: \sum_{n\in \mathcal{A}^*_{p_1\cdots p_l}} \sum_{t\mid n} \lambda(t)\ll \sum_{\substack{X^{\eta_0}\leq p_1\leq \cdots\leq p_l \\ t\leq Y}}^{\sim}\ \sum_{n\in \mathcal{A}^*_{[p_1\cdots p_l, t]}} 1  \tag{6.30}
\]

From Lemma \ref{lem54} we have
\[
\# \mathcal{A}^*_{[p_1\ldots p_l, t]} = O\left( \frac{\#\mathcal{A}^*}{[p_1\cdots p_l, t]} \right)  \tag{6.31}
\]
The major arcs estimate for Lemma \ref{rlem616} can now be concluded:\\
From (6.27), (6.28), (6.29), (6.30) and (6.31) we obtain:
$$\frac{1}{X}\sum_{q\leq Q_0}\sum_{(c,q)>1}\sum_{\frac{a}{X}\in I_{c,q}(L_0)} \left( E_{0, \mathcal{A}^*, \lambda}\left(\frac{a}{X}\right) - \frac{\kappa_{\mathcal{A}}\#\mathcal{A}^*}{\#\mathcal{B}^*} E_{0, \mathcal{B}^*, \lambda}\left(\frac{a}{X}\right) \right) S_{c_0}^2\left(\frac{a}{X}\right) e\left(-N_0\frac{a}{X}\right)$$
$$ = O\left(  (\#\mathcal{A}^*) X(\log X)^{-A}\right)\:.$$
The minor arcs estimates are now obtained by treating the $\mathcal{A}$-part $E_{0, \mathcal{A}^*, \lambda}(\theta)$ and the $\mathcal{B}$-part $E_{0, \mathcal{B}^*, \lambda}(\theta)$ separately. The estimates are easily carried out by the application of Lemmas \ref{lem617}, \ref{lem619}, \ref{lem620}, \ref{lem621} and \ref{lem622}.

\textit{Conclusion of the proof of Proposition \ref{prop41}}

\noindent The next step consists in replacing the functions $\lambda^{\pm}$ from Lemma \ref{rlem616} by the M\"onius function $\mu$, thus obtaining the exponential sum $E_0(\theta, \eta)$ from Proposition \ref{prop41}. We set
$$\mathcal{U}^{*'}:=\{ m\in \mathcal{U}^*\::\: (m, 10)=1\}$$
$$\mathcal{B}^{*'}:=\{ n\in \mathcal{B}^*\::\: (n, 10)=1\}$$
and observe that
$$S(\mathcal{U}^{*'}, X^{\eta_0}, \theta, \mu)=S(\mathcal{U}^*, X^{\eta_0}, \theta, \mu)$$
and
$$S(\mathcal{B}^{*'}, X^{\eta_0}, \theta, \mu)=S(\mathcal{B}^*, X^{\eta_0}, \theta, \mu)\:.$$
because of the condition $\lambda(t)=0$ for $(t, 10)>1$ we have:
\[
J(E(\mathcal{A}^*, X^{\eta_0}, \lambda^-))\leq J(E(\mathcal{A}^*, X^{\eta_0}, \mu))\leq J(E(\mathcal{A}^*, X^{\eta_0}, \lambda^+)) \tag{6.32}
\]
and
\[
J(E(\mathcal{B}^*, X^{\eta_0}, \lambda^-))\leq J(E(\mathcal{B}^*, X^{\eta_0}, \mu))\leq J(E(\mathcal{B}^*, X^{\eta_0}, \lambda^+)). \tag{6.33}
\]
We now apply Lemma \ref{rlem61} with 

\begin{eqnarray}
g(p):=\left\{ 
  \begin{array}{l l}
   0\:, & \quad \text{if $p\in \{2, 5\}$}\vspace{2mm}\\ 
    1/p\:, & \quad \text{otherwise}\:,\\
  \end{array} \right.
\nonumber
\end{eqnarray}

and obtain:\\
For all $\epsilon>0$ there is an $\eta^*$, such that
$$\limsup_{k\rightarrow\infty} \frac{|J(E(\theta, \eta^*))|\log X}{|\mathcal{A}^*|X}< \epsilon\:,\ \ \text{for}\ \eta^*\leq \eta_0\:.$$

We still have to pass from $X^{\eta_0}$ to $X^{\theta_2-\theta_1}$. We modify the analysis in \cite{maynard}, p. 156:\\
Given a set $\mathcal{C}$ and an integer $d$ we let
$$T_m(\mathcal{C}; d, \theta):=\sum_{\substack{X^\eta\leq p_m'\leq \cdots \leq p_1'\leq X^\theta \\ d{p_1'\cdots p_m'}\leq X^{\theta_1}}} S(\mathcal{C}_{p_1'\cdots p_m'}, X^\eta, \theta) $$
$$U_m(\mathcal{C}; d, \theta):=\sum_{\substack{X^\eta\leq p_m'\leq \cdots \leq p_1'\leq X^\theta \\ d{p_1'\cdots p_m'}\leq X^{\theta_1}}} S(\mathcal{C}_{p_1'\cdots p_m'}, p_m'X^\eta, \theta) $$
$$V_m(\mathcal{C}; d, \theta):=\sum_{X^\eta< p_m'\leq \cdots \leq p_1'\leq X^\theta } S(\mathcal{C}_{p_1'\cdots p_m'}, p_m', \theta) .$$
Buchstab's identity shows that 
$$U_m(\mathcal{C}; d, \theta)=T_m(\mathcal{C}; d, \theta)-U_{m+1}(\mathcal{C}; d, \theta)-V_{m+1}(\mathcal{C}; d, \theta)$$
The $T_m$-terms are now handled by Lemma $\ref{rlem616}$, whereas the $V_m$-terms are reduced to Proposition \ref{prop42}.\\
Proposition \ref{prop41} now has been reduced to Proposition \ref{prop63} and Proposition \ref{prop42}.

\textit{Proof of Proposition \ref{prop42} assuming Proposition \ref{prop64}}

\noindent We restrict ourselves to $E_1(\theta, \eta_0)$, since the case of $E_2(\theta, \eta_0)$ is completely analogous. As in the proof of Proposition \ref{prop41} we replace the variable factoris $e(n(\frac{c}{q}+\xi))$ by $e(n_0\xi)e(n\frac{c}{q})$ with $n_0\in\mathcal{B}^*$.\\
We obtain
$$\sum_{0, \mathcal{A}^*}e\left(\frac{c}{q}+\xi\right)=\sum_{n}1_{U(\mathcal{A}^*_{p_1\ldots p_l}, p_j)}(n)\:e\left(n\left(\frac{c}{q}+\xi\right)\right)=\Sigma_{0, \mathcal{A}^*}^{(1)}  +\Sigma_{0, \mathcal{A}^*}^{(2)}  $$
with
$$\Sigma_{0, \mathcal{A}^*}^{(1)}= e(n_0\xi) \sum_{n}1_{U(\mathcal{A}^*_{p_1\ldots p_l}, p_j)}(n) e\left(n\frac{c}{q}\right)$$
and
$$\Sigma_{0, \mathcal{A}^*}^{(2)}= e(n_0\xi) \sum_{n}1_{U(\mathcal{A}^*_{p_1\ldots p_l}, p_j)}(n) ( e(n\xi)-e(n_0\xi)).$$
An analogous decomposition holds for
$$\sum_{0, \mathcal{B}^*}e\left(\frac{c}{q}+\xi\right)\:.$$
The claim of Proposition \ref{prop42} now follows quite analogously to the proof of Lemma \ref{rlem616}. We use Proposition \ref{prop64} for the estimate of
$$\Sigma_{0, \mathcal{A}^*}^{(1)}  -\frac{\kappa_{\mathcal{A}}\#\mathcal{A}^*}{\#\mathcal{B}^*}  \Sigma_{0, \mathcal{B}^*}^{(1)}\:, $$
where for the other major arcs contribution we again use the estimate 
$$|e(n\xi)-e(n_0\xi)|=O(|n-n_0|\:|\xi|)\:.$$
The minor arcs estimates follow again by the application of Lemmas \ref{lem617}, \ref{lem619}, \ref{lem620}, \ref{lem621} and \ref{lem622}.

\textit{Proof of Proposition \ref{prop43}}

\noindent We first deal with  the $a$-variable major arcs contribution:\\
Let $1\leq q\leq Q_0$, $(c, q)=1$, $\eta=q^{-1}X^{-1}L_0$. By Lemma \ref{lem62} and the GRH we have for $|\xi|\leq \eta$:
$$S_{c_0}\left(\frac{c}{q}+\xi\right)=\frac{\mu(q)}{\phi(q)}\:\sum_{m\in Int(N_0)} e(m\xi)+O(X^{1-\delta_0}).$$

We now approximate $E\left(\frac{c}{q}+\xi\right)$. For $n\in \mathcal{U}\left( \mathcal{B}^*_{\prod(\vec{p})}, X^{z(Log \vec{p})}\right)$ we write
$$n=p_1\cdots p_l\cdot m\ \ \text{and}\ \ m=q_1\cdots q_r$$
with
$$X^{z(Log \vec{p})} \leq q_1< q_2<\cdots <q_v.$$
By partitioning the range of the $p_k$  and the $q_j$ into intervals and using GRH we see
\begin{align*}
\mathcal{U}(q, s)&:=\# \left\{ n\in \mathcal{U}\left( \mathcal{B}^*_{\prod(\vec{p})}, X^{z(Log \vec{p})}\right),\ n\equiv s \bmod q\right\}  \tag{6.34}\\
&=\mathcal{U}(q, s_0)\left(1+O((\log X)^{-A})\right)\:,
\end{align*}
for any $s_0$ with $(q, s_0)=1$, i.e. $\mathcal{U}(q, s)$ is asymptotically independent of $s$.\\
From (6.34) we obtain:
\begin{align*}
E\left(\frac{c}{q}+\xi\right)&=\sum_{\substack{s \bmod q \\ (s, q)=1}} e\left(\frac{sc}{q}\right) \sum_{n\in \mathcal{U}(q, s)} e(n\xi)(1+O(\log X)^{-4})\\
&=\frac{\mu(q)}{\phi(q)} \sum_{\vec{p}\::\: Log\: \vec{p}\:\in\: \mathcal{R}}\ \  \sum_{m\:\in\: \mathcal{U}(\mathcal{B}_{\prod(\vec{p})}, X^{z(Log\: \vec{p})})}  (1+O(\log X)^{-C_4}).
\end{align*}

We obtain

\begin{align*}
&\int_{\frac{c}{q}-\eta}^{\frac{c}{q}+\eta} E\left(\frac{c}{q}+\xi\right) S_{c_0}\left(\frac{c}{q}+\xi\right)^2 e\left(-N_0\left(\frac{c}{q}+\xi\right)\right) d\xi  \tag{6.35}\\
&=\frac{\mu(q)^3}{\phi(q)^3} e\left(-N_0\frac{c}{q}\right) \int_{-1/2}^{1/2} E(\xi)\sum_{(n_1, n_2)\in Int} e(n_1\xi+n_2\xi) e(-N_0\xi) d\xi\\
&=\frac{\mu(q)^3}{\phi(q)^3}\ e\left(-N_0\frac{c}{q}\right) \sum_{\vec{p}\::\: Log\: \vec{p}\:\in\: \mathcal{R}} \#\Big\{(m, n_1, n_2)\::\: m\in \mathcal{U}\left( \mathcal{B}^*_{\prod(\vec{p})}, X^{z(Log\: \vec{p})}\right),\\
&\ \ \ \  n_i\in Int, m+n_1+n_2=N_0\Big\}\:.
\end{align*}

We write $m=p_1\cdots p_l\cdot h$ with $m=q_1\cdots q_v$. By the well-known connection between the Buchstab function and the number of integers free of small prime factors, we have:
\begin{align*}
&\#\left\{ h\::\: h\in \frac{\# \mathcal{B}^*}{p_1\cdots p_l}\::\: p(h\Rightarrow p\geq z(Log(\vec{p}))\right\} \tag{6.36}\\
&=\frac{\# \mathcal{B}^*}{p_1\cdots p_l} \omega\left( \frac{\log (X/p_1\cdots p_l)}{Log(\vec{p})}\right) \frac{1}{\log X} (1+o(1))\:.
\end{align*}
The function
$$M(q)=\sum_{(c, q)=1}e\left(-N_0\frac{c}{q}\right)$$
is a multiplicative function of $q$: We obtain the singular series $\mathfrak{S}(N_0)$.\\
From (6.34), (6.35), (6.36) we obtain the major arcs contribution:
\begin{align*}
&\sum_{q\leq Q_0} \sum_{(c, q)=1} \int_{\frac{c}{q}-\eta}^{\frac{c}{q}+\eta} E(\xi)S_{c_0}^2(\xi) e(-N_0\xi)d\xi\\
&= \frac{X(\# \mathcal{B}^*)}{4\log X} \mathfrak{S}_0(N_0) \int\cdots\int_\mathcal{R} \frac{\omega(1-u_1-\cdots - u_l)}{u_1\cdots u_l z(u_1, \ldots, u_l)} du_1\cdots du_l (1+o(1))\:.
\end{align*}

The proof of Proposition \ref{prop43} is complete by application of Lemmas \ref{lem617}, \ref{lem619}, \ref{lem620}, \ref{lem621}, \ref{lem622}.

%


\section{Sieve asymptotics for local version of Maynard}

In this section we prove Proposition \ref{prop63}. We also prove Proposition \ref{prop64} assuming Proposition \ref{prop72} given below.

\textit{Proof of Proposition \ref{prop63}}

Let $q=q_1q_2$ with $q_1\mid 10$, $(q_2, 10)=1$. The solution set of the congruence condition 
$$n\equiv s \bmod q$$
is then a union of solution sets of systems of the form
\[
n\equiv u_1\bmod 10 \tag{I}
\]
\[
n\equiv u_2\bmod q_2\:, \tag{II}
\]
where $u_1\in\{1, 3, 7, 9\}$. We also have that
\[
n\equiv 0\bmod p_1\ldots p_l t\:. \tag{III}
\]
We substitute
$$u=10\tilde{u}+u_1\:.$$
The system (I), (II), (III) then becomes 
$$10\tilde{u}+u_1 \equiv v_2 \bmod [p_1\ldots p_l t, q_2]\:,$$
if (I), (II), (III) are compatible.\\
The system (I), (II), (III) may be written with the use of exponential sums
\begin{align*}
&\sum_{X^{\eta_0}\leq p_1\leq \cdots \leq p_l} S\left(\mathcal{A}^*_{p_1\cdots p_l, q, s}, X^{\eta_0}, \lambda, \frac{c}{q}\right)  \tag{7.1}\\
=& \sum_{\substack{s \bmod q_2 \\ (s, q_2)=1}} \ \sum_{X^{\eta_0}\leq p_1 \leq \cdots \leq p_l}\ \sum_{t\leq q} \lambda(t) \sum_{\tilde{n}\in\frac{\mathcal{B}^*-U}{10}} \frac{1}{[t, q_2]} \sum_{l=0}^{[t, q_2]-1} e\left( \frac{l(\tilde{u}-10^{-1}(q_2)(v_2-u_i)}{[t, q_2]}  \right)\:,
\end{align*}
where $10\cdot 10^{-1}(q_2) \equiv 1 \bmod q_2$.\\
We now carry out the same computation with $\mathcal{B}^*$ instead of $\mathcal{A}^*$. We find that the main terms cancel and the other ones may be estimated by Lemma \ref{lem54} to give the main result.\\
We now shall reduce Proposition \ref{prop64} to Proposition \ref{prop73} stated below. The ranges of summation (6.4) and (6.5) are defined by several sets of linear forms of the vectors
\[
Log^{(1)} (n) :=\left( \frac{\log \tilde{p}_1}{\log X} , \ldots,  \frac{\log \tilde{p}_v}{\log X}  \right)\:, \tag{7.2}
\]
where $n=\tilde{p}_1\cdots \tilde{p}_v$.\\
1) The linear forms from $\mathcal{L}$, included by $\sum^\sim$\:.\\
2) The linear forms related to the conditions 
$$n\in \mathcal{A}^*_{p_1\cdots p_v}\:,\ \ p\mid n\ \Rightarrow \ p>p_j\:.$$
3) The linear forms related to the chain of inequalities 
$$p_1\leq \cdots \leq p_v\:.$$
4) The linear forms analogous to (3) related to the other prime factors.

\noindent All the linear forms from (1) to (4) now form a set
\[
\mathcal{L}^*:=\bigcup_v \tilde{\mathcal{L}}(v)\:, \tag{7.3}
\]
where $v$ denotes the total number of prime factors.\\
To be able to describe the set of integers satisfying these linear inequalities by a polytope, we pass from the vector $Log^{(1)}$ in (7.2) to the vector
\[
Log^{(2)} (n) :=\left( \frac{\log \tilde{p}_1}{\log n} , \ldots,  \frac{\log \tilde{p}_v}{\log n}  \right)\:. \tag{7.3}
\]
Obviously
$$Log^{(2)} (n)\in Q_v(\eta):= \{(x_1, \ldots, x_n)\in\mathbb{R}^v\:,\ \eta\leq x_1\leq \cdots \leq x_v, x_1+\cdots+x_v= 1\}\:.$$
By a closed convex polytope in $\mathbb{R}^v$ we mean a region $R$ defined by a finite number of non-affine linear inequalities in the coordinates (equivalently this is the convex hull of a finite set of points in $\mathbb{R}^v$).\\
Given a closed convex polytope $R\subseteq Q_l(\eta)$, we let

\begin{eqnarray}
1_R(n):=\left\{ 
  \begin{array}{l l}
   1\:, & \quad \text{if $n=p_1\cdots p_v$ with $Log^{(2)}(n)\in R^v$}\vspace{2mm}\\ 
    0\:, & \quad \text{otherwise}\:,\\
  \end{array} \right.
\nonumber
\end{eqnarray}

We now let $\widebar{R}\subseteq [\eta, 1]^{v-1}$ denote the projection of $R$ onto the first $v-1$ coordinates (which is also a convex polytope).

\begin{definition}\label{def71}
Fix $\eta>0$ and let $v\in\mathbb{Z}$ satisfy $1\leq v\leq 2/q$. Let $\gamma>0$ and let
$$\vec{a}:=(a_1, a_2, \ldots, a_{v-1})$$
be a sequence of real numbers. Let 
$$\vec{p}:=(p_1, \ldots, p_v)$$
be an $l$-tuplet of prime numbers, $\prod(\vec{p})=p_1\cdots p_v$. Then we define
$$\mathcal{C}(\vec{a}, \gamma):= \left\{\vec{p}=(p_1, \ldots, p_v)\::\: p_j\in (X^{a_j}, X^{a_j+\gamma}), 1\leq j\leq v_1, \prod(\vec{p})\in\mathcal{B}^*\right\}$$
and 
$$\mathcal{C}(\vec{a}, \gamma, q, s):=\left\{\vec{p}\in \mathcal{C}(\vec{a}, \gamma)\::\: \prod(\vec{p})\equiv s \bmod q\right\}\:.$$
The sequence $\vec{a}$ and the box $\mathcal{C}(\vec{a}, \gamma)$ are called \textit{normal}, if $a_j+\gamma<a_{j+1}$, for \mbox{$1\leq j\leq v-2$.}
\end{definition}

\begin{proposition}\label{prop72}
Let $\mathcal{C}(\vec{a}, \gamma)$ be as defined in Definition \ref{def71}, $\gamma=(\log X)^{-C_3}$ for $C_3>0$ fixed. Let $q\leq Q_0$, $(s, q)=1$. Then
$$\sum_{n\in \mathcal{C}(\vec{a}, \gamma, q, s)} w_n=O\left(\left(\frac{1}{\phi(q)} \sum_{n\in\mathcal{C}(\vec{a}, \gamma, q, s)} 1\right) (\log X)^{-A}\right)\:.$$
\end{proposition}

\textit{Proof of Proposition \ref{prop64} assuming Proposition \ref{prop73}}

\begin{definition}\label{prop73}
Let $\delta_0:=(\log X)^{-C_3}$. We cover  $[\eta, 1]^{v-1}$ by $O_\eta(\delta_0^{-(v-1)})$ disjoint hypercubes $\mathcal{C}(\vec{a}, \delta_0)$. We partition the $\vec{a}\in\widebar{R}$ into two disjoint sets:
$$\mathcal{Y}_1:=\{ \vec{a}\in\widebar{R}\::\: \mathcal{C}(\vec{a}, \delta_0)\subseteq \widebar{R}\}$$
$$\mathcal{Y}_2:=\{ \vec{a}\in\widebar{R}\::\: \mathcal{C}(\vec{a}, \delta_0)\cap bd \widebar{R}\neq \emptyset \}\:.$$
Since the set $\mathcal{L}^*$ of linear forms defining $R$ imply
$$\frac{\log p_i}{\log n}\neq \frac{\log p_j}{\log n}\:,\ \text{for}\ i\neq j\:,$$
$\mathcal{C}(\vec{a}, \gamma)\subseteq R$ implies that $\mathcal{C}(\vec{a}, \gamma)$ is normal. 
\end{definition}

\noindent We have thus by Proposition \ref{prop72} that
$$\sum_{\vec{a}\::\: \mathcal{C}(\vec{a}, \delta_0)\subseteq \widebar{R}}\  \sum_{n\in \mathcal{C}(\vec{a}, \delta)} e\left(n \frac{c}{q}\right) w_n^*=\sum_{\substack{s \bmod q \\ (s, q)=1}} e\left(\frac{sc}{q}\right)\ \sum_{\vec{a}\::\:\mathcal{C}(\vec{a}, \delta_0)}\ \sum_{n\in \mathcal{C}(\vec{a}, \delta_0, q, s)} w_n^*$$

\noindent By the Prime Number Theorem for short intervals and arithmetic progressions, we have for any $s_0$ with $(s_0, q)=1$:
$$\sum_{n\in \mathcal{C}(\vec{a}, \delta_0, q, s)}1 =\left(\sum_{n\in \mathcal{C}(\vec{a}, \delta_0, q, s_0)}1 \right)\left( 1+O(\log X)^{-A}\right)\:.$$
Thus, we obtain by Proposition \ref{prop72}:

$$\sum_{\vec{a}\::\: \mathcal{C}(\vec{a}, \delta_0)\in \widebar{R}}\ \sum_{n\in C^+(u, \gamma)} e\left(n\frac{c}{q}\right) w_n^*=O\left(\sum_{\vec{a}\::\: \mathcal{C}(\vec{a}, \delta_0)\subseteq \widebar{R}} |\mathcal{C}(\vec{a}, \delta_0)|\right) (\log X)^{-A}\:.$$ 

\noindent For the contribution of $\mathcal{Y}_2$ we estimate the total volume of the $\mathcal{C}(\vec{a}, \delta_0)$ and treat the $\mathcal{A}^*$-part and the $\mathcal{B}^*$-part separately.\\
Proposition \ref{prop64} thus has been reduced to Proposition \ref{prop72}.


\section{$b$-variable circle method}

In this section we state propositions needed in the estimate of type II expressions by the $b$-variable circle method. We then derive Proposition \ref{prop72} from them.

\begin{proposition}\label{prop81}
Fix $\eta>0$ and let $v\in \mathbb{Z}$ satisfy $1\leq v\leq 2/\eta$. Let 
$$\mathcal{C}:=\mathcal{C}(\vec{a}, r, q, s)$$
be as in Definition \ref{def71}. Let $q\leq Q_0$. Let $\mathcal{M}^{(b)}=\mathcal{M}^{(b)}(\mathcal{C}_4)$ be given by
$$\mathcal{M}^{(b)}:=\left\{ 0\leq b< X\::\: \left|\frac{b}{X}-\frac{d}{r}\right|\leq \frac{(\log X)^{C_4}}{X} \right\} $$
for some integers $d, r$ with $r\leq (\log X)^{C_4}$, $r\mid X$.\\
Then, if $C_4$ is chosen sufficiently large,
$$\frac{1}{X}\sum_{\substack{0\leq b< X \\ b\in \mathcal{M}}} \mathcal{S}_{\mathcal{A}^*} \left(\frac{b}{X}\right)S_{\mathcal{C}}\left(-\frac{b}{X}\right)-\frac{\kappa_{\mathcal{A}}\#\mathcal{A}^*}{\#\mathcal{B}^*} \ \#\mathcal{C}(\vec{a}, r, q, s)=O\left(\frac{\# \mathcal{A}^*}{(\log X)^A} \right)\:.$$
The implied constants depend on $A$, but not on $\eta$, $v$ and the $a_j$.
\end{proposition}

\begin{proposition}\label{prop82}(Generic minor arcs)\\
Let $\mathcal{C}$ and $\mathcal{M}(C_4)$ be as in Proposition \ref{prop81}. Then there is some exceptional set
$$\mathcal{E}:=\mathcal{E}(\mathcal{C}) \subseteq [0,X]\:,\ \ \text{with}\ \ \#\mathcal{E}\leq X^{23/40}\:,$$
such that
$$\frac{1}{X}\sum_{\substack{b< X \\ b\not\in \mathcal{E}}} \left| S_{\mathcal{A}^*}\left(\frac{b}{X}\right) S_{\mathcal{C}}\left(-\frac{b}{X}\right)  \right| =O\left( \frac{\# \mathcal{A}^*}{X^\epsilon} \right)  \:. $$
The implied constant depends on $\eta$ but not on the $a_j$.
\end{proposition}

\begin{proposition}\label{prop83}(Exceptional minor arcs)\\
Let $\mathcal{C}$ and $\mathcal{M}=\mathcal{M}(C_4)$ be as given in Proposition \ref{prop81}. Let $a_1, \ldots, a_{v-1}$ in the definition of $\mathcal{C}(\vec{a}, r, q, s)$ satisfy
$$\sum_{i\in \mathcal{I}} a_i \in \left[  \frac{9}{40}+\frac{\epsilon}{2},\ \frac{16}{25}-\frac{\epsilon}{2} \right] \cup  \left[  \frac{23}{40}+\frac{\epsilon}{2},\ \frac{16}{25}-\frac{\epsilon}{2} \right] $$
for some $\mathcal{I}\subseteq \{1, \ldots, v-1\}$ and let $C_4$ be sufficiently large. Let $\mathcal{E}\subseteq [0, X]$ be any set, such that $\#\mathcal{E}\leq X^{23/40}$. Then we have
$$\frac{1}{X}\sum_{\substack{b\in \mathcal{E} \\ b\not\in \mathcal{M}}} S_{\mathcal{A}^*} \left(\frac{b}{X}\right) S_{\mathcal{C}}\left(-\frac{b}{X}\right)=O\left( \frac{\#\mathcal{A}^*}{(\log X)^A} \right)\:.$$
The implied constant depends on $\eta$ but not on the $a_1, \ldots, a_{v-1}$.
\end{proposition}

\textit{Proof of Proposition \ref{prop72}}

\noindent By orthogonality we have
$$\#(\mathcal{C}\cap \mathcal{A}^*)=\frac{1}{X}\sum_{1\leq b\leq X} S_{\mathcal{A}^*} \left(\frac{b}{X}\right) S_{\mathcal{C}}\left(-\frac{b}{X}\right)\:.$$

Proposition \ref{prop72} now follows by the partition given by Proposition \ref{prop81}, \ref{prop82} and \ref{prop83}.


\section{$b$-variable Major Arcs}

In this section we establish Proposition \ref{prop82}. We split $\mathcal{M}^{(b)}$ up as three disjoint sets.

$$\mathcal{M}^{(b)}= \mathcal{M}_1\cup \mathcal{M}_2 \cup \mathcal{M}_3\:,$$
where
$$\mathcal{M}_1:=\left\{b\in \mathcal{M}^{(b)}\::\: \left|\frac{b}{X}-\frac{d}{r}\right|\leq \frac{(\log X)^{C_2}}{X}\ \text{for some}\ d, r\leq (\log X)^{C_3}, \ r\nmid X \right\}\:,$$

$$\mathcal{M}_2:=\left\{b\in \mathcal{M}^{(b)}\::\: \frac{b}{X}=\frac{d}{r}+v\ \text{for some}\ d, r \leq (\log X)^{C_3},\ r\mid X,\ 0<|v|\leq \frac{(\log X)^{C_3}}{X}\right\} \:,$$

$$\mathcal{M}_3:=\left\{b\in \mathcal{M}^{(b)}\::\: \frac{b}{X}=\frac{d}{r}+v\ \text{for some}\ d, r \leq (\log X)^{C_3},\ r\mid X\right\} \:.$$

By Lemma \ref{lem54} and recalling $X$ is a power of 10, we have
$$\sup_{b\in \mathcal{M}_1} \left| S_{\mathcal{A}^*}\left(\frac{b}{X}\right) \right|=\#\mathcal{A}^* \: \sup_{b\in \mathcal{M}_1} F_{10^{k-H}}\left(\frac{b}{X}\right) = O\left( \#\mathcal{A}^*\exp(-(\log X)^{-1/2+\epsilon})\right)\:.$$

Using the trivial bound 
$$S_{\mathcal{C}(\vec{a}, r, q, s)}=O(X(\log X)^B)$$ 
and noting that 
$$\mathcal{M}_1\ll (\log X)^{3B}\:,$$
we obtain
\[
\frac{1}{X}\sum_{b\in \mathcal{M}_1} S_{\mathcal{A}^*}\left(\frac{b}{X}\right) S_{\mathcal{C}}\left(-\frac{b}{X}\right)=O\left( \frac{\#\mathcal{A}^*}{(\log X)^A} \right)\:. \tag{9.1}
\]
This gives the result for $\mathcal{M}_1$. We now consider $\mathcal{M}_2$.\\
For $\vec{p}=(p_1, \ldots, p_v)$ we write $\vec{p}_{v-1}=(p_1, \ldots, p_{v-1})$, $\prod_{v-1}(\vec{p})=p_1\cdots p_{v-1}$:
$$S_{\mathcal{C}}=\sum_{\substack{\vec{p}_{v-1}=(p_1, \ldots, p_{v-1}) \\ p_j\in (X^{a_j}, X^{a_j+\gamma})}} \sum_{\substack{\prod(\vec{p}_{v-1})p_v \in \mathcal{B}^* \\  \prod(\vec{p}_{v-1})p_v \equiv s\bmod q }} e\left(\frac{b\prod_{v-1}(\vec{p})p_v}{X}\right)\:.$$

We note that if $b\in \mathcal{M}_2$, then
$$\frac{b}{X}=\frac{d}{r}+\frac{c}{X}\:,\ \ \text{for some integers}\ b, r, |c|\leq (\log X)^{C_4}\:,\ \text{($c$ is an integer since $r\mid X$)}\:.$$
We now chose $C_5>0$, $C_5\in\mathbb{Z}$, so large, that - after $C_1,\ldots, C_4$ have been chosen - the following considerations are true and set
$$\Delta:=\lceil \log X \rceil^{-C_5}\:.$$
We remark that $\Delta^{-1}$ is an integer.\\
We separate the sum $S_{\mathcal{C}}\left(\frac{b}{X}\right)$ by putting the prime variable $p_v$ in short intervals of length
$$\Delta(\#\mathcal{B}^*)/(p_1\cdots p_{v-1})$$
and in arithmetic progressions $\bmod [q, r]$. Thus, we have
$$\left|S_\mathcal{C}\left(\frac{b}{X}\right) \right| =\sum_{\vec{p}_{v-1}\::\: p_j\in (X^{a_j}, X^{a_j+\gamma})} \sum_{p_v<\frac{|\mathcal{B}^*|}{p_1\cdots p_{v-1}}} e(p_1\cdots p_{v-1}p_v)\:.$$

If $mp= j\Delta x+O(\Delta x)$ and $p\equiv u \bmod d$, then we have
$$e\left(mp\left(\frac{d}{r}+\frac{c}{X}\right)\right)=e\left(\frac{dum}{r}\right) e(jcx)+O(\Delta (\log X)^{C_4})\:.$$
By the Prime Number Theorem in short intervals and arithmetic progressions, we have
$$\sum_{p\in [j\Delta X/m, (j+1)\Delta X/m]}1= E\frac{\Delta X}{m}(1+O((\log X)^{-A})\:,$$
where $E=1$ if the system
\begin{eqnarray}
\left\{ 
  \begin{array}{l l}
   p\equiv u \bmod r \\ 
   p\equiv s \bmod q\:,\\
  \end{array} \right.
\nonumber
\end{eqnarray}
is solvable and $E=0$ otherwise,\\
with 
$$E\frac{\Delta X}{m}(1+O((\log X)^{-A})\leq \Delta |\mathcal{B}^*|\sup_{\substack{d\leq (\log X)^C \\ r\leq (\log X)^C}} \sum_{u \equiv l\bmod q}e\left(\frac{dum}{r}\right) \sum_{1\leq j < \Delta^{-1}} e(j\Delta c)\:. $$
We have
$$\sum_{1\leq j < \Delta^{-1}} e(j\Delta c)=e(-c)=-1=O(1)\:.$$
We finally obtain
$$\frac{1}{X}\sum_{b\in \mathcal{M}_2} S_{\mathcal{A}^*}\left(\frac{b}{X}\right) S_{\mathcal{C}}\left(-\frac{b}{X}\right)=O\left(  \frac{\#\mathcal{A}^*}{(\log X)^A}\right)\:,$$
where the implied elements depend on $\eta$ and $\gamma$, but not on the $a_j$.\\
Finally we consider $\mathcal{M}_3$:\\
For $(d, r)=1$ we have:
$$S_{\mathcal{C}}\left(\frac{d}{r}\right)=\sum_{0\leq u\leq r} e\left(\frac{du}{r}\right) \sum_{\substack{n\in \mathcal{B}^*\\ n\equiv u (\bmod r)\\ n\equiv s (\bmod q) }} 1 =\frac{1}{\phi([q, r])}\left( \sum_{n\in\mathcal{L}} 1 \right) \sum_{\substack{0< u< r \\ (u, r)=1 \\ r\equiv s (\bmod (q, r))}} e\left(\frac{du}{r}\right)\:.$$ 

The solution set of 
\begin{eqnarray}
\left\{ 
  \begin{array}{l l}
   n\equiv u \bmod r \\ 
   n\equiv s \bmod q\:,\\
  \end{array} \right.
\nonumber
\end{eqnarray}
is non-empty if and only if for the square-free kernels $r_0$ of $r$ the solution set of 
\begin{eqnarray}
\left\{ 
  \begin{array}{l l}
   n\equiv u \bmod r_0 \\ 
   n\equiv s \bmod q\:,\\
  \end{array} \right.
\nonumber
\end{eqnarray}
is non-empty.\\
For the exponential sum
$$\sum_{\substack{0< u< r \\ (u, r)=1}}e\left(\frac{du}{r}\right)$$ we have:
$$\sum_{g=0}^{\frac{r}{r_0}-1}e\left( \frac{d(s+gr_0)}{r}  \right)=e\left(\frac{du}{r}\right)\sum_{g=0}^{\frac{r}{r_0}-1}e\left(\frac{g}{r/r_0}\right)
=
\left\{ 
  \begin{array}{l l}
   0, &\text{if}\ r_0<r \\ 
   1, &\text{if}\ r_0=r \\ 
  \end{array} \right.
\nonumber$$
We finally obtain:
$$S_{\mathcal{C}(\vec{a}, r, q, s)}\left(\frac{d}{r}\right)=\sum_{\substack{0< u< r \\ (u, r)=1}}e\left(\frac{du}{r}\right)\frac{\phi(q)}{\phi([q, r])}
\sum_{n\in \mathcal{C}(\vec{a}, r, q, s)} 1_{\mathcal{C}(\vec{a}, r, q, s)}(u)\left(1+O((\log X)^{-A})  \right)$$
$$= \frac{\phi(q)}{\phi([q, r])} \sum_{n\in  \mathcal{C}(\vec{a}, r, q, s)} 1 \sum_{\substack{0< u< r \\ (u, r)=1}}e\left(\frac{du}{r}\right)\left(1+O((\log X)^{-A})  \right)$$
$$=\mu(r)  \frac{\phi(q)}{\phi([q, r])} \sum_{n\in  \mathcal{C}(\vec{a}, r, q, s)} 1\:.$$
Since $\mu(r)=0$ for $r\mid 10^k$, unless $r\in \{1, 2, 5, 10\}$ the estimate can easily be concluded.


\section{Generic Minor Arcs}

In this section we establish Proposition \ref{prop82} and obtain some bounds on the exceptional set $\mathcal{E}$ by using the estimates of Lemma \ref{lem54}.

\begin{lemma}\label{lem101}
Let $\mathcal{C}=\mathcal{C}(\vec{a},\gamma, q, s)$ as in Definition \ref{def71}. We have that
$$\#\left\{ 0\leq b< X\::\: \left|S_c\left(\frac{b}{X}\right)\right| \sim \frac{X}{C}  \right\} \ll  \frac{C^2 |\mathcal{C}|}{X}\:.$$
\end{lemma}
\begin{proof}
We have
$$\sum_{b\::\: \left|S_{\mathcal{C}}\left(\frac{b}{X}\right)\right|^2\geq \frac{|\mathcal{C}|^2}{10 C^2}}  \left|S_{\mathcal{C}}\left(\frac{b}{X}\right)\right|^2  
\geq  \frac{|\#\mathcal{C}|^2}{10 C^2}  \: \#\left\{b\::\: \left|S_{\mathcal{C}}\left(\frac{b}{X}\right)\right| \geq \frac{\#\mathcal{C}}{10 C} \right\}\:.
$$
Thus
$$\#\left\{b\::\: |S_{\mathcal{C}(\vec{a}, \gamma)}|\geq \frac{|\mathcal{C}|}{10 C}  \right\}\leq \frac{10 C^2}{X^2} \sum_{b\leq X}  \left|S_{\mathcal{C}}\left(\frac{b}{X}\right)\right|^2 =\frac{10 C^2}{X^2}\: X|\mathcal{C}|\:,$$
the last identity following by Parseval's equation.
\end{proof}

\begin{lemma}\label{(Generic frequency bounds)}$ $\\
Let 
$$\mathcal{E}:=\left\{0\leq b\leq X \::\: F_X\left(\frac{b}{X}\right)\geq \frac{1}{X^{23/80}}  \right\}\:.$$
Then 
$$\#\mathcal{E} \ll X^{23/40-\epsilon}\:,$$
$$\sum_{b\in \mathcal{E}} F_X\left(\frac{b}{X}\right) \ll X^{23/80-\epsilon}\:,$$
and
$$\frac{1}{X} \sum_{\substack{b< X \\ b\not\in \mathcal{E}}} \left| F_X\left(\frac{b}{X}\right) S_{\mathcal{C}}\left(-\frac{b}{X}\right)  \right| \ll \frac{1}{X^\epsilon} \:.$$
\end{lemma}
\begin{proof}
The first bound on the size of $\mathcal{E}$ follows from using Lemma \ref{lem53} with \mbox{$B=X^{23/80}$} and verifying that
$$\frac{23\time 235}{80\times 154}+\frac{54}{433}<\frac{23}{40}\:.$$
For the second bound we see from Lemma \ref{lem54} that
\begin{align*}
\sum_{b\in\mathcal{E}} F_X\left(\frac{b}{X}\right) &\ll \sum_{\substack{j\geq 0 \\ 2^j\leq X^{23/80}}} \# \left\{ 0\leq b <X\::\: F_X\left(\frac{b}{X}\right) \sim 2^{-j} \right\}\\
&\ll \sum_{\substack{j\geq 0 \\ 2^j\leq X^{23/80}}} 2^{(235/154-1)j} X^{59/433} 
\ll X^{59/433+(23\times 235)/(80\times 154)-23/80}\:,
\end{align*}
and so the calculation above gives the result.\\
It remains to bound the sum over $b\not\in\mathcal{E}$. We divide the sum into $O((\log X)^2)$ subsums, where we restrict to these $b$, such that
$$ F_X\left(\frac{b}{X}\right)  \sim \frac{1}{B}\ \ \text{and}\ \ 
 \left|S_{\mathcal{C}}\left(\frac{b}{X}\right)\right| \sim \frac{|\mathcal{C}|}{C}$$
 for some $B\geq X^{23/80}$ and $C\leq X^2$ (terms with $C> X^2$ make a contribution O(1/X)). This gives
 $$ \frac{1}{X} \sum_{\substack{b< X \\ b\not\in \mathcal{E}}} \left|F_X\left(\frac{b}{X}\right) S_{\mathcal{C}}\left(-\frac{b}{X}\right) \right|$$ 
 $$\ll \sum_{\substack{X^{23/80}\leq B \\ 1\leq C\leq X^2}} \frac{(\log X)^2}{X} \ \sum_{\substack{b< X \\ F_X\left(\frac{b}{X}\right)\asymp \frac{1}{B} \\  S_{\mathcal{C}}\left(-\frac{b}{X}\right)\sim \frac{X}{C} }} \left|F_X\left(\frac{b}{X}\right) S_{\mathcal{C}}\left(-\frac{b}{X}\right) \right| +\frac{1}{X^2}\:. $$
 We concentrate on the inner sum:\\
 Using Lemmas \ref{lem53} and \ref{lem101} we see, that the sum contribution 
 $$\ll \frac{\#\mathcal{C}}{BC} \#\left\{ b\::\: F_X\left(\frac{b}{X}\right)\asymp \frac{1}{B}\:, \ S_{\mathcal{C}}\left(-\frac{b}{X}\right)\sim   \frac{X}{C}\right\}$$ 
 $$\ll \frac{\#\mathcal{C}}{BC}\min\left(C^2, B^{235/154} X^{50/435}\right) \ll X|\mathcal{C}|X^{2\epsilon} $$
Here we used the bound $\min(x, y)\leq X^{1/2}Y^{1/2}$ in the last line. In particular, we see this is $O(X^{1-2\epsilon})$ if $B\geq X^{23/80}$ on verifying that
$$\frac{23}{8}\times\frac{73}{308}>\frac{59}{866}\:.$$
Substituting this into our bound above gives the result.
\end{proof}


\section{Exceptional minor arcs}

\begin{lemma}\label{lem111}(Bilinear sum bound)$ $\\
Let $N, M, R\geq 1$ and $E$ satisfy
$$X^{9/25}\leq N\leq X^{17/40},\ R\leq X^{1/2}, NM\leq 1000X,\ \text{and}\ E\leq 100\:\frac{X^{1/2}}{R}$$
$$\text{and either}\ E\geq \frac{1}{X}\ \text{or}\ E=0.$$
Let $\mathcal{F}:=\mathcal{F}(R, E)$ be given by
$$\mathcal{F}:=\left\{ b<X\::\:\frac{b}{X}=\frac{d}{r}+v\ \text{for some $(d,r)=1$ with $r\asymp R$, $v=\frac{E}{X}$}\right\}\:.$$
Then for any 1-bounded complex sequences $\alpha_n, \beta_n, \gamma_b$ we have
$$\sum_{b\in \mathcal{F}\cap \mathcal{E}} \sum_{\substack{n\sim N \\ m\sim M}} \alpha_n \beta_m \gamma_b e\left(-\frac{bnm}{X}\right) \ll \frac{X(\log X)^{O(1)})}{(R+E)^{\epsilon/10}}\:.$$
\end{lemma}
\begin{proof}
This is Lemma 13.1 of \cite{maynard}.
\end{proof}

We now derive Proposition \ref{prop83} from Lemma 13.1.


\textit{Proof of Proposition \ref{prop83}}.\\
By symmetry, we may assume that $\mathcal{I}=\{1,\ldots, l_1\}$ for some $l_1<l$. By Dirichlet's theorem on Diophantine approximation, any $b\in [0,X]$ has a representation
$$\frac{b}{X}=\frac{d}{r}+\nu$$
for some integers $(d,r)=1$ with $r\leq X^{1/2}$ and some real $|\nu|\leq 1/X^{1/2}r$.\\
Thus we can partition $[0,X]$ into $O((\log X)^2)$ sets $\mathcal{F}(R,E)$ as defined by Lemma \ref{lem101} for different parameters $R,E$ satisfying 
$$1\leq R\leq X^{1/2}\ \text{and}\  E=0\ \text{or}\ \frac{1}{X}\leq E\leq \frac{100X^2}{R}\:.$$
Moreover, if $b\not\in \mathcal{M}^{(b)}$, then $b\in \mathcal{F}=\mathcal{F}(R,E)$ for some $R,E$ with 
$$R+E\geq  (\log X)^{C_3}\:.$$
Thus, provided $C_3$ is sufficiently large, we see that it is sufficient to show that 
\[
\frac{1}{X} \left| \sum_{b\in\mathcal{F}\cap \mathcal{E}} S_{\mathcal{A}}\left(\frac{b}{X}\right) S_{\mathcal{C}_{q,s}}\left(-\frac{b}{X}\right)   \right| \ll \frac{\#\mathcal{A}}{(R+E)^{\epsilon/20}}\:.    \tag{10.1}
\]
Recalling the Definition \ref{def71} 
$$\mathcal{C}_{q,s}:=\{ \vec{p}=(p_1, p_2, \ldots, p_l)\::\: p_i\in I_i\:, \Pi(\vec{p})\equiv s\: (\bmod q)\}$$
let
$$\mathcal{C}^{(I)}:=\bigtimes_{j\in \mathcal{I}} I_{i_j}\:,\ \mathcal{C}^{(II)} :=\bigtimes_{j\not\in \mathcal{I}} I_{i_j}\:, $$
such that
$$n\in \mathcal{C}_i^{(I)}\ \Rightarrow\ X^{9/25}\leq n\leq X^{17/40}.$$
We have (with $t^{-1}t\equiv 1\: (\bmod q)$):
$$\mathcal{C}_{q,s}=\bigcup_{\substack{ t\mod q \\ (t,q)=1 }} (\mathcal{C}_{q,t}^{(I)}\times \mathcal{C}_{q, t^{-1}s}^{(II)})$$
and thus
\begin{align*}
&\frac{1}{X}\sum_{n\in \mathcal{E}} S_{\mathcal{A}}\left(\frac{b}{X}\right)S_{\mathcal{C}_{q,s}}\left(-\frac{b}{X}\right)\\
&\ll \sum_{\substack{t \bmod q \\ (t,q)=1}}\ \ \sum_{b\in \mathcal{F}\cap \mathcal{E}} S_{\mathcal{A}}\left(\frac{b}{X}\right)    \sum_{\substack{n_1\sim N_1 \\ n_2\sim N_2}} \alpha_{n_1}\beta_{n_2}\: e\left(-\frac{bn_1n_2}{X}\right)\:,
\end{align*}
where
\begin{eqnarray}
\alpha_{n_1}:=\left\{ 
  \begin{array}{l l}
    1\:, & \quad \text{if $n_1\in \mathcal{C}_{q,t}^{(I)}$}\vspace{2mm}\\ 
    0\:, & \quad \text{otherwise}\:,\ \ \ \ 
  \end{array} \right.
\nonumber
\beta_{n_2}:=\left\{ 
  \begin{array}{l l}
    1\:, & \quad \text{if $n_2\in \mathcal{C}_{q,t^{-1}s}^{(II)}$}\vspace{2mm}\\ 
    0\:, & \quad \text{otherwise}\:.\\
  \end{array} \right.
\nonumber
\end{eqnarray}

Thus it suffices to show that 
\[ 
\frac{1}{X}\:\sum_{b\in\mathcal{F}\cap \mathcal{E}}S_{\mathcal{A}}\left(\frac{b}{X}\right) \sum_{n\sim N} \alpha_n \sum_{m\sim M} \beta_m\: e\left(-\frac{bnm}{X}\right)\ll \frac{\#\mathcal{A}^*}{(\log X)^{A}},   \tag{10.2}
\]
Let $\gamma_b$ be the 1-bounded sequence, satisfying 
$$S_{\mathcal{A}}\left(\frac{b}{X}\right)=\#\mathcal{A}\gamma_b F_X\left(\frac{b}{X}\right).$$
After substituting this expression for $S_{\mathcal{A}}$, we see that (10.2) follows immediately from Lemma \ref{lem101}, if the parameter $C_3$  is chosen sufficiently large.\\

\noindent\textbf{Acknowledgements}.\\
M. Th. Rassias: I would like to express my gratitude to Artur Avila and Ashkan Nikeghbali for their essential support throughout the preparation of this work.

\end{document}